\documentclass{article}

\usepackage{style}

\title{Harmonic maps to Hadamard spaces and a universal higher Teichm\"{u}ller space}
\author{J. Maxwell Riestenberg and Peter Smillie}
\date{\today}

\begin{document}

\maketitle

\begin{abstract}
    We give a sufficient criterion, which we call stability, for a coarse Lipschitz map $f$ from a complete manifold $X$ with Ricci curvature bounded below to a proper Hadamard space $Y$ to be within bounded distance of a harmonic map. We prove uniqueness of the harmonic map under additional assumptions on $X$ and $Y$.

    Using this criterion, we prove a significant generalization of the Schoen-Li-Wang conjecture on quasi-isometric embeddings between rank 1 symmetric spaces. In particular, under a natural generalization of the quasi-isometric condition, we remove the assumption that the target has rank 1. 
    This allows us to define a universal Hitchin component for each $\PGL_d(\R)$, generalizing universal Teichm\"uller space, and show that it can be described both as a space of quasi-symmetric positive maps from $\RP^1$ to the flag variety, and as a space of harmonic maps.

\end{abstract}

\tableofcontents

\section{Introduction}

Let $X$ and $Y$ be complete non-compact Riemannian manifolds. A natural problem is the following, posed for instance by Gromov in \cite[Section 2.B]{GromovFoliatedPlateau}:

\begin{problem} \label{pseudoproblem}
    Given a smooth map $f: X \to Y$, does there exist a harmonic map $h: X \to Y$ which is at bounded distance from $f$, i.e.
    \[
        d_Y(f(x),h(x)) \leq C < \infty \textrm{ ?}
    \]
\end{problem}   

A very reasonable additional condition to impose is that $f$ be coarse Lipschitz, meaning for some $L$, and all $x,z \in X$,
\begin{equation} \label{eqn: coarse Lipschitz}
    d_Y(f(x),f(z)) \leq L(d_X(x,z) + 1).
\end{equation}
Indeed, coarse Lipschitz maps are exactly those which preserve the relation of bounded distance in the sense that $f \circ g \sim f \circ g'$ whenever $g \sim g'$. In addition, if $X$ is the universal cover of a compact manifold $\mathcal{X}$, any locally bounded map $f: X \to Y$ that is equivariant under an action of $\pi_1(\mathcal{X})$ on $Y$ is coarse Lipschitz.

It is also natural to require that the Ricci curvature of $X$ is bounded below,
\begin{equation}
    \mathrm{Ric}(X) \geq -Kg_X.
\end{equation}
Besides the fact that it holds when $X$ is the universal cover of a compact manifold, it is also necessary for any uniformity of estimates on harmonic maps.

More restrictively, we require that the sectional curvature of $Y$ is non-positive. This still includes the important case of symmetric spaces of non-compact type, and it is also the setting that Gromov works in. 

Finally, for us it is not terribly important that $Y$ be a manifold. 
Using the theory of harmonic maps to Hadamard spaces of Korevaar and Schoen \cite{Korevaar:1993aa}, we assume only that $Y$ is a proper Hadamard space, i.e.\ a complete, locally compact CAT(0) metric space. 
In this context, we consider measurable maps instead of smooth ones; we will use map throughout this paper to mean a measurable map. 
All told, we have refined Problem \ref{pseudoproblem} to:

\begin{problem} \label{problem}
    Let $X$ be a complete manifold with Ricci curvature bounded below, and $Y$ a proper Hadamard space. Given a coarse Lipschitz map $f: X \to Y$, when does there exist a harmonic map at bounded distance from $f$?
\end{problem}

Some partial answers to this problem are understood under the assumption of some negative upper bound on the sectional curvature of $Y$ or strong assumptions on the derivative of $f$, e.g.\ \cite[Section 2.B]{GromovFoliatedPlateau}. More recently, a sufficiently strong condition in the setting of strictly negative curvature led to the resolution by Markovic \cite{Markovic} of the well-known Schoen conjecture, and in a more general form by Benoist and Hulin \cite{BH1,BH2}. 
We will return to this story in Section \ref{sec: intro schoen}.

In the present work, we propose a condition on $f$, which we call stability, and prove that it is sufficient for a positive answer to Problem \ref{problem}. Crucially, it is applicable even without any strict negative curvature assumption on $Y$, and so we are able to apply it to produce new classes of harmonic maps into higher rank symmetric spaces.

To motivate the definition of stability, we recall three basic facts. First, since $Y$ is a Hadamard space, the distance $d_y$ to any point $y\in Y$ is a convex function on $Y$. Second, if $h: X \to Y$ is a harmonic map, it pulls back convex functions to subharmonic functions. Third, subharmonic functions on $X$ are those that are weakly increasing under the heat flow operator
\[
P_t u (x) = \int_X u(z) p_t(x,z)dz
\]
where $p_t(x,z)$ is the heat kernel. 

\begin{definition} \label{def: intro stable}
    A coarse Lipschitz  map $f: X \to Y$ is \emph{stable} if there exists $t > 0$ such that for all $x \in X$ and $y \in Y$,
    \begin{equation}
        P_t (d_y \circ f) (x) - d_y \circ f(x) \geq 1
    \end{equation}
\end{definition}

This is a variant of the Definition \ref{def: stable at scale r} of stability below. We find Definition \ref{def: stable at scale r} to be more practical, and it may be more general, but it is slightly harder to appreciate at first blush. We remark that stability is a coarse notion, in the sense that if $f$ is stable, then so is any map at bounded distance from $f$.

Our main theorem is

\begin{theorem}[Theorem \ref{thm: Donalson-Corlette}] \label{thm: main Donaldson-Corlette}
     Let $X$ be a complete manifold with Ricci curvature bounded below, and $Y$ a proper Hadamard space. Every stable coarse Lipschitz map from $X$ to $Y$ is at bounded distance from a harmonic map.
\end{theorem}

One of the key ideas which lets us avoid a strict negative curvature assumptions on $Y$ is the Ptolemy inequality of planar geometry, which we now recall. Given any four points in the Euclidean plane, let $D,D'$, $E,E'$, and $F,F'$ be the pairwise distances between them, grouped in opposite pairs. The following two inequalities always hold:
\begin{equation}
    F^2 + (F')^2 \leq D^2 + (D')^2 + E^2 + (E')^2 
\end{equation}
and
\begin{equation}
    FF' \leq DD' + EE'.
\end{equation}
The first is called the parallelogram inequality, and the second the Ptolemy inequality. Equality for the first is obtained, naturally, for parallelograms, and equality for the second is obtained for quadrilaterals inscribed in the circle, where in both cases $F$ and $F'$ are the diagonals. Both inequalities hold also for any four points in a CAT(0) space \cite{Reshetnyak1968non-expansive}.

In the regularity theory for harmonic maps to Hadamard spaces developed by Korevaar and Schoen, the parallelogram inequality plays an essential role. 
To prove regularity, they compare the harmonic map $h$ with the nearly-harmonic map $h'= h \circ \phi$ for a small diffeomorphism $\phi$, and apply the inequality to the four points $h(x),h(z),h'(x),h'(z)$. The symmetries of the parallelogram suggest that this is the right inequality for comparing two of the same sort of map, for example two nearly-harmonic maps. On the other hand, in the present work we need to compare a harmonic map with a coarse-Lipschitz map, and this asymmetry is reflected in the asymmetry of inscribed quadrilaterals. Indeed the quadrilateral that we study looks somewhat more like an inscribed quadrilateral than a parallelogram (see Figure \ref{fig:thinking diagram}).

Although we do not explicitly assume the presence of strict negative curvature in $Y$, there is a weak type of non-flatness implicit in the definition of stable. 
Namely, if $Y$ is isometric to $Y_0 \times \R$, then no map to $Y$ can be stable, which one can see by considering $d_y$ for $y \to \pm \infty$ in the $\R$ direction. 
There is also a type of weak negative curvature condition on $X$ implicit in stability, in the sense that unless $\lambda_1(X) > 0$, there can be no stable maps from $X$. Indeed, this follows from the existence of any coarse Lipschitz function $u$ on $X$ satisfying $P_t u - u \geq 1$ for some $t$.

\subsection{Uniqueness, and equivariant maps} \label{sec: intro uniqueness}

We now restrict to the case that $X$ is connected with a cocompact isometry group, and that $Y$ is a symmetric space of non-compact type. 
For the following uniqueness theorem, we emphasize that we assume no equivariance of the stable map.

\begin{theorem}[Theorem \ref{thm: uniqueness}] \label{thm: main uniqueness}
    Suppose that $X$ is connected and cocompact and $Y$ is a symmetric space. 
    Every stable map from $X$ to $Y$ is at bounded distance from a \emph{unique} harmonic map.
\end{theorem}

Here the assumption that $X$ admits a cocompact action by isometries is not very important, but allows us to avoid discussing metric-measure spaces with lower Ricci curvature bounds. 
The assumption on $Y$ is also stronger than it needs to be, but some extra condition on $Y$ is essential, analogous to Labourie's condition `sans demies-bandes plates' in \cite{Labourie1991existenceharmoniques}.

We now turn to equivariant maps from $X$ to $Y$. 
Let $\Gamma$ act cocompactly by isometries on $X$, and fix an action $\rho$ of $\Gamma$ on $Y$. 
One can construct a locally bounded $\rho$-equivariant map from $X$ to $Y$, and it is easy to see that such a map will be coarse Lipschitz and that any two $\rho$-equivariant locally bounded maps will be at bounded distance from each other.

\begin{theorem}[Theorem \ref{thm: equivariant}] \label{thm: main equivariant}
    Suppose that $G$ is a semisimple Lie group and $Y$ is the associated symmetric space of non-compact type. Suppose $X$ is connected and $\Gamma$ acts cocompactly by isometries on $X$.
    If $f\colon X \to Y$ is equivariant under a Zariski-dense action $\rho$ of $\Gamma$ on $Y$, then $f$ is stable.
\end{theorem}

Recall that the theorem of Donaldson \cite{Donaldson} and Corlette \cite{Corlette} gives in this setting
\begin{enumerate}
    \item There exists a $\rho$-equivariant harmonic map if and only if the Zariski closure of $\rho$ is reductive.
    \item The equivariant harmonic map is unique if and only if the centralizer of $\rho$ is compact.
\end{enumerate}

When the symmetric space is $\SL_n(\C)/\mathrm{SU}(n)$, the centralizer condition corresponds to the representation on $\C^n$ being irreducible. 
More generally, it is natural from the invariant-theoretic perspective to call this condition \emph{stable}. 
Since stability is a coarse notion, it therefore makes sense to talk about the stability of an action $\rho$ as the stability of such an equivariant map. 
Although it may appear unfortunate at first that this is weaker than our definition of stable, this issue is both minor and unavoidable, see Remark \ref{rmk: LK}.

We note that Labourie extended the definition of reductive to actions on an arbitrary Hadamard manifold $Y$, and proved an analog of the Donaldson-Corlette theorem in that setting \cite{Labourie1991existenceharmoniques}, see also \cite{jostyau1991harmonic}.

The proof of our Theorem \ref{thm: main equivariant} uses the Donaldson-Corlette theorem, so we are not giving a new proof of existence in the equivariant case. Instead, the main point of Theorem \ref{thm: main equivariant} is to give evidence that our notion of stability is a natural criterion in the context of Problem \ref{problem}.

\subsection{The Schoen-Li-Wang conjecture} \label{sec: intro schoen}

An asymptotic Dirichlet problem asks for a harmonic map between two spaces with prescribed boundary values `at infinity'. In general, it is not always clear what boundary values at infinity should mean. One way of making this precise is through the formulation in Problems \ref{pseudoproblem} and \ref{problem}. Of course, this is somewhat unsatisfying in the sense that the coarse map $f$ is certainly not itself `at infinity'.

There is a special setting in which one can formulate the asymptotic Dirichlet problem directly. The visual boundary $\partial_\infty Y$ of a non-positively curved space $Y$ is the space of geodesic rays in $Y$, in which two rays are equivalent if they lie at bounded distance. If $X$ is also non-positively curved, a coarse Lipschitz map from $X$ to $Y$ may or may not induce a map from $\partial_\infty X$ to $\partial_\infty Y$, but in the following setting it will.

A map $f: X \to Y$ between metric spaces is called a quasi-isometric embedding if in addition to being coarse Lipschitz \eqref{eqn: coarse Lipschitz} for some $L$,
\[
    d_X(x,z) \leq M(d_Y(f(x),f(z)) + 1)
\]
holds for some $M$. 
If $X$ and $Y$ are hyperbolic in the sense of Gromov, then quasi-isometric embeddings between $X$ and $Y$ extend continuously to the visual boundary \cite{Gromov1987hyperbolicgroups}, and moreover are determined up to bounded distance by their boundary maps \cite{BonkSchramm}.
Furthermore, the class of boundary maps $\phi\colon  \partial_\infty X \to \partial_\infty Y$ that one obtains can be characterized in terms of how they distort cross-ratios; they are called quasi-symmetric maps. 
We will return to the notion of quasisymmetry in section \ref{sec: intro uhteich}.

The Schoen conjecture \cite{Schoenconjecture} asserts that every quasi-isometric embedding from the hyperbolic plane $\H^2$ to itself is within bounded distance of a unique harmonic map. In other words, the conjecture is a positive answer to Problem \ref{problem}, together with a uniqueness statement, for this class of maps. By the above discussion, this can be reformulated as saying that every quasi-symmetric map from $\H^2$ to $\H^2$ has a unique quasi-isometric harmonic extension. Li and Wang \cite{liHarmonicRoughIsometries1998} conjectured the same when $\H^2$ is replaced by a symmetric space of rank 1.

The Schoen conjecture was one of the motivations behind Gromov's paper \cite{GromovFoliatedPlateau}, and it continued to provoke much work over the intervening 25 years \cite{Markovic2,Markovic3,liHarmonicRoughIsometries1998, HardtWolf,BonsanteSchlenker,Tam-Wan} until it was finally resolved in the affirmative by Markovic in 2015 \cite{Markovic}. Shortly after, Benoist and Hulin proved that every quasi-isometric embedding between (possibly different) rank 1 symmetric spaces is within bounded distance of a harmonic map \cite{BH1}. They later extended this further \cite{BH2}, replacing rank 1 symmetric spaces by pinched Hadamard manifolds, and relaxing somewhat the conditions on $f$. The theorem for quasi-isometric embeddings was later extended in \cite{SidlerWenger2021harmonicqie} to allow $Y$ to be a locally compact Gromov hyperbolic Hadamard space. Further extensions were made by To\v{s}i\'{c} \cite{Tosic,TosicII} and Pankka-Souto \cite{pankka2017harmonicextensionsquasiregularmaps}.

Although the strict negative curvature assumptions in the works of Markovic and Benoist-Hulin allow for more techniques than we have access to in our setting, our general methods recover most of their results. In particular, using Theorem \ref{thm: main Donaldson-Corlette} we simplify the proof of the main theorem of \cite{BH1}.

\begin{theorem}[{Corollary \ref{cor: benoist-hulin}, c.f.\ \cite[Theorem 1.1]{BH1} for existence, \cite[Theorem 2.3]{liHarmonicRoughIsometries1998} for uniqueness}]
    Let $X$ and $Y$ be rank 1 symmetric spaces, and $f: X \to Y$ a quasi-isometric embedding. 
    Then there is a unique harmonic map at bounded distance from $f$.
\end{theorem}

\subsection{A higher rank Schoen conjecture}

The second main result of this paper, and our original inspiration, is an extension of the Schoen conjecture to the case that the target is a symmetric space of higher rank. Let us explain what we mean by this. 
Let $G$ be a semi-simple Lie group, and let $Y$ be the associated symmetric space of non-compact type.
Choose a Cartan subspace $\mathfrak{a}$, a set $\Pi$ of simple restricted roots and let $\mathfrak{a}^+ \subset \mathfrak{a}$ be the closed positive Weyl chamber.
There is a natural vector-valued distance $\vec{d}_Y\colon Y \times Y \to \mathfrak{a}^+$, from which we recover $d_Y$ by taking its Euclidean norm. 
For any nonempty subset $\Theta \subset \Pi$, we say a map $f\colon X \to Y$ is a $\Theta$-quasi-isometric embedding if it is coarse Lipschitz and satisfies for some $M$
\begin{equation}\label{eqn: intro theta-qie}
    d_X(x,z) \leq M(\alpha \circ \vec{d}_Y(f(x),f(z)) + 1) \quad \textrm{for all } \alpha \in \Theta.
\end{equation}
It is easy to see that this is equivalent to the definition of uniformly regular quasi-isometric embeddings of Kapovich-Leeb-Porti \cite{kapovich2018morse}. 
It should also be compared with the uniformly hyperbolic bundles of Bridgeman-Labourie \cite{bridgeman2025ghostpolygonspoissonbracket}. 
When $f$ is the orbit map of a representation $\rho$ of a finitely generated group, (\ref{eqn: intro theta-qie}) is equivalent to $\rho$ being $\Theta$-Anosov \cite{kapovich2018morse, bochi2019anosov}.

As a consequence of their higher rank Morse lemma, Kapovich, Leeb and Porti prove that every $\Theta$-quasi-isometric embedding from a proper geodesic metric space $X$ to $Y$ has Gromov hyperbolic domain and extends to a continuous boundary map from $\partial_\infty X$ to the partial flag variety $\mathcal{F}_\Theta$ corresponding to $\Theta$ \cite{kapovich2018morse}.

We give a criterion for stability in terms of the boundary map $\partial f$ induced by a $\Theta$-quasi-isometric embedding $f$. 
We call a limit of $f$ a locally uniform limit of $f$ under recentering, and we say that a boundary map is finitely non-transverse if it transverse to any given flag except at finitely many points. 
As an example of the criterion, when $\Theta = \Pi$ is the set of all simple roots, we obtain
\begin{theorem}[Corollary \ref{cor: positive separation and finitely non-transverse limits}] \label{thm: intro criterion for stability}
    If $f \colon X \to Y$ is a Lipschitz $\Pi$-quasi-isometric embedding such that every limit of $f$ has finitely non-transverse boundary map, then $f$ is stable.
\end{theorem}

Our criterion is fairly flexible, but is not satisfied by all boundary maps of $\Theta$-quasi-isometric embeddings.
It would be interesting to study in full generality when it holds. 
In this paper, we focus on a special case in which it does hold: when $Y = Y_d$, the symmetric space for $\mathrm{PGL}_d(\R)$, $\Theta = \Pi$, the full set of roots, and the boundary map $\RP^1 \to \mathrm{Flag}(\R^d)$ is positive in the sense Fock-Goncharov \cite{fock2006moduli}.
This means that every triple of flags in its image is equivalent, up to the action of $\mathrm{PGL}_d(\R)$ to the triple $\{\sigma_0, n \cdot \sigma_0, \sigma_\infty\}$ where $\sigma_0$ is the standard descending flag, $\sigma_\infty$ the standard ascending flag, and $n$ is an upper triangular unipotent matrix which is totally positive \cite{schoenberg1930variationsvermindernde, GantmacherKrein1935oscillatoires, lusztig1994positivity} i.e.\ all of the interesting minors of $n$ are positive. 
It would be natural to try to extend the theorem to the more general setting of $\Theta$-positive maps as introduced by Guichard-Wienhard \cite{gw2025positivity}.

This allows us to conclude our second main theorem

\begin{theorem}[Corollary \ref{cor: schoen}] \label{thm: main schoen}
    For $d \geq 2$, let $Y_d$ be the symmetric space of $\mathrm{PGL}_d(\R)$. 
    Every $\Pi$-quasi-isometric embedding $\H^2 \to Y_d$ with positive boundary map is within bounded distance of a unique harmonic map.
\end{theorem}

When $d = 2$, a $\Pi$-quasi-isometric embedding is the same as a quasi-isometric embedding, and the positivity condition is vacuous. 
Hence, this is exactly the statement of the Schoen conjecture when $d=2$. 
For $d > 2$, our main interest in this theorem is that it shows that a certain universal higher Teichm\"uller space can be described as a space of harmonic maps. 
We describe this further in the following section.

\subsection{A universal higher Teichm\"uller space} \label{sec: intro uhteich}

The universal Teichm\"uller space was originally defined as the space of boundary values of quasi-conformal maps of the disk \cite{Bers}. 
It is universal in the sense that it contains the Teichm\"uller space of any closed surface. 

By work of many people including \cite{BeurlingAhlfors1956quasiconformal,DouadyEarle1986natural,Tukia1985extension, Morse1924lemma}, these are the same as the boundary maps of quasi-isometric embeddings from $\H^2$ to itself, i.e.\ the quasi-symmetric maps. 
We recall that the cross ratio of four points in $\RP^1$ is 
\[
    CR(\lambda_1,\lambda_2,\lambda_3,\lambda_4) = \frac{(\lambda_4 - \lambda_3)(\lambda_2 - \lambda_1)}{(\lambda_3 - \lambda_2)(\lambda_1 - \lambda_4)}
\]
and a map $\phi$ from $\partial_\infty \H^2 \cong \RP^1$ to itself is quasisymmetric if \[
CR(\phi(\lambda_1),\phi(\lambda_2),\phi(\lambda_3),\phi(\lambda_4))
\]
can be bounded uniformly in terms of $CR(\lambda_1,\lambda_2,\lambda_3,\lambda_4)$. 
Another rephrasing of the Schoen conjecture is that the universal Teichm\"uller space admits a description as a space of harmonic maps.

A different sort of generalization of the classical Teichm\"uller space is the notion of a higher Teichm\"uller space as defined in \cite{WienhardInvitation}. 
This is a union of components of the space of representations of a closed surface group consisting entirely of discrete and faithful representations. 
The archetype of a higher Teichm\"uller space is, for each $d \geq 2$, the space of Hitchin representations of a closed surface group into $\mathrm{PGL}_d(\R)$. 
When $d = 2$, these are exactly the Fuchsian representations. 
For $d > 2$, these components were first described by Hitchin \cite{HitchinLieGroups}, and proved to be higher Teichm\"{u}ller spaces in this sense by Labourie \cite{LabourieAnosovFlows} and Fock-Goncharov \cite{fock2006moduli}.
Of relevance to this paper, their first description by Hitchin used equivariant harmonic maps, through the theory of Higgs bundles.

Associated to each Hitchin representation is an equivariant boundary map from $\RP^1$ to the space $\mathrm{Flag}(\R^d)$ of full flags of $\R^d$.
Such a map is a hyperconvex Frenet curve in the language of Labourie, or equivalently, a positive map in the language of Fock-Goncharov. 
This property also characterizes Hitchin representations \cite{Guichard2008hyperconvexes,fock2006moduli}.
With this in mind, we define
\begin{definition} \label{def: UHitd}
    The \emph{universal Hitchin component $\mathrm{UHit}_d$} is the space of $\Pi$-quasi-isometric embeddings from $\H^2$ to $Y_d$ with positive boundary map from $\RP^1$ to $\mathrm{Flag}(\R^d)$, up to bounded distance.
\end{definition}
In these terms our Theorem \ref{thm: main schoen} says that the universal Hitchin component is naturally a space of harmonic maps.

By the discussion preceding the definition, for a fixed $d$, the $d$-Hitchin component of any closed surface group embeds into to $\mathrm{UHit}_d$. 
To complete the analogy between $\mathrm{UHit}_d$ and universal Teichmuller space, and to recast our existence Theorem \ref{thm: main schoen} as an asymptotic Dirichlet problem, it remains to characterize $\mathrm{UHit}_d$ as a space of positive maps from $\RP^1$ to $\mathrm{Flag}(\R^d)$ which are moreover quasi-symmetric in some sense.

Both Fock-Goncharov \cite{fock2006moduli} and Labourie \cite{LabourieAnosovFlows} have suggested defining a universal Hitchin component in terms of positive maps of circles to $\Flag(\R^d)$. 
Each paper describes functions of $k$-tuples of flags that generalize the cross-ratio; using the respective notions of cross ratio one obtains a priori different notion of quasi-symmetric maps from $\RP^1$ to $\mathrm{Flag}(\R^d)$. 
We also remark that Labourie and Toulisse define a slightly different sort of universal higher Teichm\"uller space containing the spaces of maximal representations into $\SO(2,n+1)$ \cite{labourie2022quasicirclesquasiperiodicsurfacespseudohyperbolic}, still in terms of cross ratios. 
Whereas ours consists of quasisymmetric maps, theirs consists of quasi-circles. 
Using results of \cite{LabourieToulisseWolf}, they show that their universal higher Teichm\"{u}ller space parametrizes a space of quasi-isometrically embedded maximal surfaces in the pseudo-Riemannian symmetric space $\mathbb{H}^{2,n}$. 

We define a cross ratio which is similar to, but different than, the edge invariant considered by Fock-Goncharov. Namely, if $n$ and $m^{-1}$ are totally positive unipotent upper triangular matrices, we define $d-1$ cross-ratios of a positive four-tuple of flags in standard position by 
\[
CR_i(m \cdot \sigma_0, \sigma_0, n \cdot \sigma_0, \sigma_\infty) = m_{i,i+1}/n_{i,i+1}
\]
for $1 \leq i \leq d-1$. 
We say a positive map $\phi: \RP^1 \to \mathrm{Flag}(\R^d)$ is quasi-symmetric if each $CR_i$ is controlled by the cross-ratio in the domain. 
We prove
\begin{theorem}[Theorem \ref{thm: from qs map to purqie} and Theorem \ref{thm: purqies with same boundary}] \label{thm: main quasisymmetric} The universal Hitchin component is equal to the space of positive quasisymmetric maps from $\RP^1$ to $\mathrm{Flag}(\R^d)$.
\end{theorem}

This allows us to express our Theorem \ref{thm: main schoen} as a solution of an asymptotic Dirichlet problem:
\begin{corollary}
    Every quasisymmetric positive map from $\RP^1$ to $\mathrm{Flag}(\R^d)$ has a unique extension to a $\Pi$-quasi-isometric harmonic map from $\H^2$ to $Y_d$.
\end{corollary}

We conclude with a few remarks about this theorem.

First, we note that in contrast to the case of strict negative curvature, the visual boundary of $Y$ does not give the right target with respect to which to pose the asymptotic Dirichlet problem. This is already clear from the case of $\H^2 \times \H^2$, for which the product structure shows that the right target for the boundary map is $\RP^1 \times \RP^1$, which is the space of full flags.

Second, although the existence theorem is much easier in the case that the $\Pi$-quasi-isometric embedding arises from a Hitchin representation, we hope that the method of bounding the distance to the harmonic map will have interesting applications in the equivariant case, for instance to the asympotic geometry of the K\"{a}hler metric on the Hitchin component. It would still take some work however to make our constants explicit.

Finally, there has also been recent progress on harmonic maps from open Riemann surfaces, such as $\H^2$, to symmetric spaces from the perspective of Higgs bundles. In \cite{li2023higgsbundleshitchinsection}, Li and Mochizuki construct a candidate for a universal Hitchin component, parametrized by the space of bounded holomorphic differentials $(q_2, \ldots, q_n)$ on $\H^2$. If one could show that their family of harmonic maps coincides with ours, this would give a non-Abelian Hodge correspondence for the universal Hitchin component.

\subsection{Acknowledgments}

Both authors would like to thank Anna Wienhard for support throughout the project, and Michael Gekhtman for pointing out the references \cite{ShapiroShapiro2000grassmannconvexityRP3,SaldanhaShapiroShapiro2021Grassmannconvexityrevisited}. The second author thanks Vladimir Markovi\'{c} for encouragement and Richard Schoen for pointing out the reference \cite{Freidin}.

\section{Preliminaries}

\subsection{Hadamard spaces} \label{sec: Hadamard}

If $T$ is a geodesic triangle in a metric space $Y$, a planar comparison triangle is a triangle $\hat{T}$ in $\R^2$ with a map to $T$ that is an isometry along each geodesic. A CAT(0) space $Y$ is a metric space in which any two points can be joined by a geodesic segment, such that for any geodesic triangle $T$ in $Y$, the canonical map from a planar comparison triangle to $T$ is 1-Lipschitz. A Hadamard space is a complete CAT(0) space. $Y$ is called proper if metric balls in $Y$ are compact. A standard reference for CAT(0) spaces is \cite{bridsonMetricSpacesNonPositive1999}.

The visual boundary $\partial_\infty Y$ of a Hadamard space $Y$ is the set of geodesic rays $\eta: [0,\infty) \to Y$ up to equivalence, where two rays at bounded distance are considered equivalent. The space $\overline{Y} = Y \cup \partial_\infty Y$ is given a topology, called the visual topology, from its inclusion into the inverse limit of closed balls in $Y$ \cite[Chapter 11.8]{bridsonMetricSpacesNonPositive1999}. 
If $Y$ is proper, then $\overline{Y}$ is compact.

If $\eta \in \partial_\infty Y$, the Busemann cocycle associated to $\eta$ is the function $b_\eta \colon Y\times Y \to \R$ defined by
\[
b_\eta(y_1, y_2) = \lim_{t \to \infty} d_Y(\eta(t), y_2) - d_Y(\eta(t), y_1)
\]
(one checks using the triangle comparison property that this does not depend on the choice of ray in the equivalence class). 
We will usually de-emphasize the point $y_1$, and speak of the Busemann function $b_\eta(y_2)$, which is only well defined up to a constant. 
The CAT(0) condition implies that the distance is convex on $Y \times Y$.
In particular each distance function $d_y \colon Y \to \R$, $d_y(y') = d_Y(y,y')$ is convex as well as each Busemann function. 

\subsubsection{Quadrilateral comparison} 

A useful property of CAT(0) spaces is Reshetnyak's sub-embedding theorem. For any geodesic quadrilateral $Q$ in a metric space $Y$, let $D,D',E,E'$ be its edges, grouped in opposite pairs, and $F,F'$ be its diagonals. 
A planar subembedding of $Q$ is a quadrilateral $\hat{Q}$ in $\R^2$ with a map to $Q$ that is an isometry along each edge, and distance non-increasing along its diagonals.

\begin{theorem}[{\cite{Reshetnyak1968non-expansive}, see also \cite{Korevaar:1993aa}}]
    Every quadrilateral in a CAT(0) space has a planar subembedding.
\end{theorem}

As one of the key ingredients in \cite{Korevaar:1993aa}, Reshetnyak's subembedding theorem is used to transport various versions of the parallelogram inequality from planar quadrilaterals to quadrilaterals in CAT(0) spaces. In the present work, the key ingredient turns out to be instead the Ptolemy inequality. For planar quadrilaterals with pairwise distances $D,D'$, $E,E'$, and $F,F'$, this reads
\[
FF' \leq DD' + EE'.
\]
In this case equality holds if and only if the quadrilateral is inscribed in a circle. It is interesting to note that the Ptolemy inequality and the parallelogram inequality are simply the triangle inequality, applied to two generators of the space of symmetric polynomials on two letters. By the subembedding theorem, as above, we have

\begin{prop}[Ptolemy inequality] \label{prop: Ptolemy}
    For any four points in a CAT(0) space, with pairwise distances $D,D'$, $E,E'$, $F,F'$ grouped in opposite pairs, it holds that
    \[
    FF' \leq DD' + EE'.
    \]
\end{prop}

It is interesting to note that a metric space is CAT(0) if and only if it satisfies the Ptolemy inequality and the distance is a convex function on $Y \times Y$ \cite{LytchakPtolemy}.

\subsection{Lower Ricci curvature bounds}

Recall that $X$ is a complete Riemannian $n$-manifold with Ricci curvature bounded below by $-K$.

\subsubsection{Green's functions and mean value inequalities}

If $\Omega \subset X$ is a bounded domain, we write $G^\Omega(x,z)$ for the Green's function of $\Omega$. The Laplacian comparison theorem and the maximum principle gives the following well-known sharp lower bound for the Green's function of a ball with respect to a point $x$:

\begin{proposition} \label{prop: noncollapsed greens function bound}
    Let $B = B(x,r)$ be a ball in $X$. Then
    \[
    G^B(x,z) \geq \Gamma(d_X(x,z)) - \Gamma(r)
    \]
    where 
    \begin{equation} \label{eqn: Gamma(r)}
    \Gamma(\rho) = \int_0^\rho \sinh(\sqrt{K/(n-1)}s)^{-(n-1)}ds
    \end{equation}
    is the Green's function of the space-form of Ricci curvature $-K$.
\end{proposition}

When $X$ is non-collapsed in the sense that the volume of balls of radius 1 in $X$ is uniformly bounded below, the following proposition follows from Proposition \ref{prop: noncollapsed greens function bound}. But it is also true more generally.

\begin{proposition}
[Green's function lower bound] \label{thm: Greens function lower bound}
    Let $B = B(x,r+2)$ be a ball in $X$, and let $G^B(x,z)$ be the Green's function of $B$ centered at $x$. There is a constant $c=c(K,n,r)$ such that for all $z \in B(x,r+1)$,
    \begin{equation} \label{eqn: G lower bound}
    G^B(x,z) \geq \frac{c}{V(x,r+1)}
    \end{equation}
\end{proposition}

\begin{proof}[Sketch of proof, compare {\cite[Theorem 6.1]{Saloff-Coste2}}]

Let $p_t$ be the heat kernel of $X$ and $p_t^B$ the heat kernel of $B$. The Li-Yau Harnack inequality \cite[Theorem 2.1]{Li-Yau} implies Gaussian decay of $p_t(x,z)$ when $t \ll r = d_X(x,z)$, and by a standard argument using stochastic completeness, this implies a lower bound for $V(x,r+1)p_t(x,x)$. The Gaussian decay, together with the parabolic maximum principle, bounds $p^B_t(x,x)$ from below in terms of $p_t(x,x)$. Finally, another application of the Li-Yau Harnack inequality bounds $p^B_{2t}(x,z)$ from below in terms of $p^B_t(x,x)$ for $z$ in the interior of $B$. Since
\[
G^B(x,z) = \int_0^\infty p^B_t(x,z) dt \geq \int_0^\epsilon p^B_t(x,z) dt
\]
this gives a lower bound for $V(x,r+1)G^B(x,z)$.

\end{proof}

Finally, we will use the following mean value inequality.

\begin{theorem}[{Mean value inequality, \cite[Lemma 1.6]{Li-Tam} and \cite[Theorem 5.1]{Saloff-Coste2}}]\label{thm: MVI}
   Suppose $B(x,r+1) \subset X$, and $u$ is a positive weak subsolution of $\Delta u \geq -K u$ on $B(x,r+1)$. Let $V(x,r+1)$ be the volume of $B(x,r+1)$. There is a constant $C=C(n,K,r)$ such that
   \[
   \sup_{\overline{B(x,r)}} u \leq \frac{C}{V(x,r+1)} \int_{B(x,r+1)} u(z) dz .
   \]
\end{theorem}

\subsubsection{Mollification}

As a consequence of the Bishop-Gromov volume comparison, we can mollify coarse Lipschitz maps to produce Lipschitz maps in a uniform way, by averaging over balls. Compare \cite[Proposition 3.4]{BH1}.

\begin{prop} \label{prop: Lipchitz}
    Suppose $X$ is a complete $n$-manifold with Ricci curvature bounded below by $-K$, $Y$ a proper Hadamard space, and $f\colon X \to Y$ a $L$-coarsely Lipschitz map. 
    There is a constant $C$ depending only on $n$ and $K$ such that there is a $CL$-Lipschitz map $f^{(1)}$ whose distance from $f$ is bounded by $L$.
\end{prop}

\begin{proof}
    We construct $f^{(1)}(x)$ by pushing forward a measure by $f$ and taking its center of mass as in \cite[Lemma 2.5]{BH1}. This construction goes back at least to Karcher \cite{Karcher} when $Y$ is a manifold. 

    For each $x \in X$, let $\mu_x$ be the volume measure on the ball of radius 1 around $x$, and define $f^{(1)}(x)$ to be the center of mass of $f_*\mu_x$. Since $f$ is $L$-coarse Lipschitz, the support of $f_*\mu$ is contain in the ball $B_L(f(x)) \subset Y$. By convexity of balls in $Y$, the center of mass is also contained in $B_L(f(x))$, so the distance from $f$ to $f^{(1)}$ is bounded by $L$. 
    
    We now estimate the Lipschitz constant of $f^{(1)}$. Let $x_1$ and $x_2$ be any two points in $X$ such that $d_X(x_1,x_2) = 2\epsilon < \frac{1}{2}$. For $i=1,2$, write $\mu_i = \mu_{x_i}$ and $y_i = f(x_i)$. Also let $x$ be a point in $X$ half-way between $x_1$ and $x_2$, meaning $d_X(x,x_1) = d_X(x,x_2) = \epsilon$.
    
    For $\nu$ any signed measure of bounded variation, write $||\nu||$ for the total variation. By Lemma 2.5 of \cite{BH2} (which clearly works for Hadamard spaces, not just Hadamard manifolds), we can estimate
    \[
    d_Y(y_1,y_2) \leq \frac{8ML}{m}
    \]
    where $m = \mathrm{min}(||f_*\mu_1||, ||f_*\mu_2||)$, $M = ||f_*\mu_1 - f_*\mu_2||$, and $L$ is the coarse Lipschitz constant of $f$.  Since pushforward does not increase total variation, and preserves it for (unsigned) measures, we have
    \[
    m = \min_i\;\mathrm{Vol}(B_1(x_i))
    \]
    \[
    M \leq \mathrm{Vol}(B_1(x_1) \triangle B_1(x_2))
    \]
    where $\triangle$ denotes the symmetric difference. Simple geometry gives
    \[
    \frac{M}{m} \leq \frac{\mathrm{Vol}(B_{1 + \epsilon}(x)) - \mathrm{Vol}(B_{1-\epsilon}(x))}{\mathrm{Vol}(B_{1 - \epsilon}(x))}
    \]
    and the Bishop-Gromov volume comparison theorem bounds this by the corresponding volume ratio in the spaceform of curvature $-K$, which gives
    \[
    \frac{M}{m} \leq \frac{C}{8} \epsilon
    \]
    where $C$ depends on $n,K$, and the fact that $\epsilon < \frac{1}{4}$. Hence $d_Y(f(x_1),f(x_2)) \leq C L d_X(x_1,x_2)$, which is what we were trying to show.
\end{proof}

\subsection{Harmonic maps} \label{sec: harmonic maps}

The basic theory of harmonic maps into Hadamard spaces is developed by Korevaar and Schoen in \cite{Korevaar:1993aa}, and also by Jost in \cite{jostEquilibriumMapsMetric1994}. Sturm \cite{Sturm} also develops this theory in even more generality. We summarize only the relevant facts for us. Let $Y$ be a Hadamard space, $(X,g)$ a Riemannian manifold, and $h: X \to Y$ harmonic in the sense of Korevaar-Schoen.

\begin{prop} \label{prop: K-S}
Given a harmonic map $h: X \to Y$, the Korevaar-Schoen energy density $e(x)$ is a locally $L^1$ function on $X$ with the following properties:
\begin{enumerate}
    \item If $Y$ is a manifold, then $e = \frac{1}{2} |\nabla h|^2$.
    \item If $y \in Y$, and $d_y(z) := d_Y(y,z)$, then 
    \[
    \Delta (d_y \circ h)^2 \geq 4 e
    \]
    weakly on $X$.
    \item If $X$ has a lower Ricci curvature bound, $\mathrm{Ric}_{ij} \geq -K g_{ij}$, then
    \[
    \Delta e \geq -K e
    \]
    weakly on $X$.
\end{enumerate}
\end{prop}

\begin{proof}
    Property (1) is proved by Korevaar-Schoen in \cite{Korevaar:1993aa}, and is included here simply for normalization. Property (2) is basically proved in \cite{Korevaar:1993aa}, and is also a special case of Theorem 1.5 in \cite{LytchakStadler}. Property (3) is the main theorem of \cite{Freidin}.
\end{proof}
Whether or not $Y$ is a manifold, we will write $|\nabla h|$ for $\sqrt{2e}$.

\medskip

A crude gradient bound for harmonic maps follows from Proposition \ref{prop: K-S}, together with the mean value inequality \ref{thm: MVI} and the Caccioppoli trick. 

\begin{prop}[Compare {\cite[Section IX.4]{SchoenYaubook}}] \label{prop: gradient estimate}
    Let $X$ be an $n$-manifold with Ricci curvature bounded below by $-K$ and $Y$ a Hadamard space, $x\in X$, $y\in Y$. 
    Suppose $h\colon B(x,r+2) \to B(y,R)$ is harmonic. 
    There is a constant $C=C(K,n,r)$ such that
    \[
    \sup_{B(x,r)} |\nabla h| \leq C R.
    \]
\end{prop}
We remark that this also follows from our Lemma \ref{lem: harnack estimate}.

\medskip

For $Y$ a Hadamard manifold, the solution of the Dirichlet problem is due to Hamilton \cite{Hamilton}. For Hadamard spaces, it is proved by Korevaar and Schoen.

\begin{proposition}[{Dirichlet problem \cite[Theorem 2.2]{Korevaar:1993aa}}]\label{prop: dirichlet problem}
    Let $Y$ be a CAT(0) space, $\Omega$ a compact Riemannian manifold with boundary, and $f: \overline{\Omega} \to Y$ a Lipschitz map. 
    There is a unique harmonic map $h: \Omega \to Y$ such that $h|_{\partial \Omega} = f|_{\partial \Omega}$.
\end{proposition}

Note that properness of $Y$ is not assumed. When $Y$ is proper, an easy consequence of this proposition is the precompactness of bounded families of harmonic maps (see \cite{Korevaar:1993aa} Theorem 1.13). We give a slight reformulation suitable to asymptotic Dirichlet problems.

\begin{proposition} \label{prop: asymptotic Dirichlet}
    Let $X$ be a Riemannian manifold, $\Omega_i$ an exhaustion of $X$, $Y$ a proper CAT(0) space, and $h_i$ a sequence of harmonic maps $\Omega_i \to Y$ that are locally bounded, uniformly in $i$. 
    There is a subsequence of the $h_i$ that converges locally uniformly to a harmonic map $h\colon X \to Y$.
\end{proposition}

\begin{proof}
    Since the $h_i$ are locally uniformly bounded, and the Ricci curvature of $X$ is locally bounded below, the gradient bound, Proposition \ref{prop: gradient estimate}, implies that they are locally uniformly Lipschitz. Since $Y$ is proper, and the $h_i$ are locally uniformly bounded, it follows from Arzela-Ascoli that the $h_i$ sub-converge locally uniformly to a map $h: X \to Y$. We need only check that $h$ is harmonic. If $K$ is any compact domain in $X$, let $\phi$ be the uniform limit of $h_i|_{\partial K}$, and $h_\phi$ be its harmonic extension, which exists by Proposition \ref{prop: dirichlet problem}. Since the distance between two harmonic maps is convex, and the $h_i$ converge uniformly to $\phi$ on $\partial K_i$, they converge uniformly to $h_\phi$ on $K$, so $h|_K = h_\phi$. Hence $h$ is harmonic.
\end{proof}

\section{Stability} \label{sec: stability}

\subsection{Coarsely strictly subharmonic functions}

\begin{definition}
    We say a measure $\mu_x$ on $X$ is harmonic with respect to a point $x$ if for any subharmonic function $u$ on $X$,
    \[
    \int_X u(z) \mu_x(z) \geq u(x).
    \]
\end{definition}

If $u$ is harmonic, then applying the inequality to $\pm u$ shows that in this case we have equality. Applying the equality to the function $u=1$, one sees that every harmonic measure is a probability measure.

One example of a harmonic measure is the heat kernel measure $p_t(x,z)dz$. Another is the hitting measure of a bounded domain $\Omega$ containing $x$. This is the unique measure that is harmonic with respect to $x$ and is supported on $\partial \Omega$. We will write it as $\mu_x^\Omega$. If the boundary of $\Omega$ is sufficiently regular, Green's formula gives 
\[
    \mu_x^\Omega= -\frac{dG}{dn} dz,
\]
where $dz$ is the area measure on the boundary of $\Omega$ and $dG/dn$ is the normal derivative of the Green's function $G=G^\Omega(x, \cdot)$ of $\Omega$ whose pole lies at $x$. The hitting measure can also be characterized in probabilistic terms as describing, for a subset $U \subset \partial \Omega$, the chance that Brownian motion starting at $x$ will first exit $\Omega$ in the set $U$. 

One can construct more harmonic measures by convolving families of harmonic measures or using the optional stopping theorem. One that we use below is the collected heat kernel measure $\mu_{x,t}^r$ which is defined by stopping Brownian motion at the minimimum of $t$ and the exit time from the ball of radius $r$ around $x$. By Green's formula, if $B = B(x,r)$ and $p^B_t$ the heat kernel of $B$ with Dirichlet boundary conditions,
    \begin{equation}
    \int_X v(z) \mu^r_{x,t}(z) = \int_{B} v(z) p^{B}_t(x,z)dz - \int_{\partial B} v(z) \left(\int_0^t \frac{d p_s^B(x,z)}{dn} ds \right) dz.
\end{equation}

\begin{definition}
    A coarse Lipschitz measurable function $u$ is coarsely strictly subharmonic at scale $r$ if there exists a measurable family of harmonic measures $\mu_x$, each with support on $\overline{B(x,r)}$, such that
    \begin{equation} \label{eqn: coarsely strictly subharmonic}
        \int_X (u(z) - u(x))\mu_x(z) \geq 1.
    \end{equation}
\end{definition}

In other words, $u$ is coarsely strictly subharmonic at scale $r$ if there is an operator $P$ such that
\begin{enumerate}
    \item $P$ is given by convolution with a family of measures,
    \item $P$ has \emph{propogation} at most $r$,
    \item $P w - w \geq 0$ for all subharmonic functions $w$, and
    \item $P u - u \geq 1$.
\end{enumerate}
In this case, we will say $u$ is coarsely strictly subharmonic with respect to the operator $P$, or to the family of measures $\mu_x$.

\medskip

The following proposition says that coarse strict subharmonicity is a coarse property.

\begin{proposition} \label{prop: subharmonicity is coarse}
    If $|u - v| \leq L$, and $u$ is coarsely strictly subharmonic at scale $r$, then $v$ is coarsely strictly subharmonic at scale $2(L + 1)r$.
\end{proposition}

\begin{proof}
Let $P$ be such that $u$ is strictly coarsely subharmonic with respect to $P$. For any positive integer $k$, $P^k$ is a convolution operator with propagation at most $kr$, which still satisfies $P^k w - w \geq 0$ for all subharmonic functions $w$, and now satisfies $P^k u - u \geq k$. Since $P^k$ is also convolution with a probability measure, 
\[
|P^k u - P^k v| \leq P^k|u-v| \leq L.
\]
Applying these observations with $k = \lceil 2L \rceil + 1$, we obtain
\[
P^k v \geq P^k u - L \geq u + L + 1 \geq v + 1
\]
and hence $v$ is coarsely strictly subharmonic at scale $2(L+1)r$.
\end{proof}

A natural measure of the degree of coarse strict subharmonicity is given by the drift.

\begin{definition} \label{def: drift}
    The drift of a function $u$ is
    \begin{equation}
        \liminf_{t \to \infty} \inf_{x \in X} \frac{(P_t u)(x) - u(x)}{t}
    \end{equation}
    where $P_t$ is the heat evolution operator.
\end{definition}

\begin{prop} \label{prop: drift implies stable}
    If $u$ has positive drift then it is coarsely strictly subharmonic at some scale.
\end{prop}

\begin{proof}[Proof sketch]
    Although $P_t$ is not of finite propagation, the lower bound on the Ricci curvature of $X$ assures that its kernel has uniform Gaussian decay, 
    \[
    p_t(x,z) \leq \frac{C}{(\mathrm{Vol}(x,\sqrt{t})\mathrm{Vol}(z,\sqrt{t}))^{1/2}} \exp\left(\epsilon t - \frac{r^2}{5t}\right)
    \]
    where $r = d_X(x,z)$ (see e.g.\ \cite[Cor 3.1]{Li-Yau}). This means that if $r \gg t$, a Brownian motion is extremely unlikely to leave the ball of radius $r$ in time $t$, and this is sufficient to show that $u$ is coarsely strictly subharmonic with respect to the collected heat kernel measure $\mu_{x,t}^r$.
\end{proof}

It would be nice to know if the converse of Proposition \ref{prop: drift implies stable} holds as well.

\subsection{Stability of maps}

We recall that $X$ is a complete Riemannian manifold with Ricci curvature bounded below by $-K$, and $Y$ is a proper Hadamard space.

\begin{definition} \label{def: stable at scale r}
    A coarse Lipschitz map $f: X \to Y$ is \emph{stable at scale $r$} if for all $y \in Y$, the function $d_y \circ f$ is coarsely strictly subharmonic at scale $r$ on $X$. We say $f$ is stable if it is stable at some scale.
\end{definition}

We note that the implicit family of measures $\{\mu_x\}$, or equivalently the operator $P$, is allowed to depend on $y$.
By Proposition \ref{prop: subharmonicity is coarse}, stability is a coarse property.

\begin{proposition} \label{prop: Busemann for free}
    If $f$ is stable at scale $r$, then $b_\eta \circ f$ is coarsely strictly subharmonic at scale $r$ for all $\eta \in \partial_\infty Y$.
\end{proposition}

\begin{proof}
    If $c: [0, \infty) \to Y$ is a geodesic ray, then by \cite[Lemma 8.18]{bridsonMetricSpacesNonPositive1999}, the distance functions $d_{c(s)}$ converge locally uniformly to the corresponding Busemann function $b$. Therefore for each $x \in X$, the integral \eqref{eqn: coarsely strictly subharmonic} for $u = d_{c(s)} \circ f$ converges to the integral for $u = b \circ f$. Hence the same inequality holds for $b \circ f$.
\end{proof}

\begin{proposition} \label{prop: Busemann enough}
    Suppose $Y$ has extendable geodesics. If $b_\eta \circ f$ is coarsely strictly subharmonic at scale $r$ for all $\eta \in \partial_\infty Y$, then $f$ is stable at scale $r$.

\end{proposition}

\begin{proof}
    Let $c$ be a geodesic ray in $Y$, and let $b \colon Y \to \R$ be the associated Busemann function, vanishing at $c(0)$. By the triangle inequality, for any $y' \in Y$, $ d(c(s),y') - s $ decreases monotonically to $b(y')$.

    Now suppose $Y$ has extendable geodesics. For any point $x \in X$ and $y \in Y$, we let $c(s)$ be a geodesic ray from $f(x)$ through $y$, with associated Busemann function $b$. Then
    \[
    d_y \circ f(x) = b\circ f(x)
    \]
    and for all $z$,
    \[
    d_y \circ f(z) \geq b\circ f(z).
    \]
    It follows that the corresponding integral \eqref{eqn: coarsely strictly subharmonic} is at least as big for $d_y \circ f$ as it is for $b \circ f$.
\end{proof}

We also define, but do not use in an essential way;

\begin{definition} \label{def: drift of a map}
    The \emph{drift} of a coarse Lipschitz map $f: X \to Y$ is
    \begin{equation}
        \liminf_{t \to \infty} \inf_{x \in X, y \in Y} \frac{P_t(d_y \circ f)(x) - d_y \circ f(x)}{t}
    \end{equation}
    where $P_t$ is the heat evolution operator.
\end{definition}

As in Proposition \ref{prop: drift implies stable}, a coarse Lipschitz map of positive drift is stable at some scale. This shows that Definition \ref{def: intro stable} in the introduction is at least as strong as Definition \ref{def: stable at scale r}.

Definition \ref{def: drift of a map} bears a strong resemblance to the conditions of the multiplicative ergodic theorem of \cite{KarlssonMargulis}.

\section{Existence} \label{sec: existence}

In this section we prove Theorem \ref{thm: main Donaldson-Corlette}. We will construct $h$ by taking a limit of a sequences of solutions to Dirichlet problems. Since $f$ itself is not assumed to be of locally finite energy, as a first step we mollify $f$, using Proposition \ref{prop: Lipchitz}, to obtain a Lipschitz map $f^{(1)}$ that is at distance at most $L$ from $f$.

Then, we construct a suitable exhaustion of $X$ by Lipschitz domains $\Omega_i$, and let $h_i \colon \Omega \to Y$ be the solution to the Dirichlet problem
\begin{equation} \label{eqn: Dirichlet problem}
    \begin{cases}
        \Delta h_i = 0 \\
        h_i |_{\partial \Omega_i} = f^{(1)} |_{\partial \Omega_i}
    \end{cases}
\end{equation}
which exists and is unique by Proposition \ref{prop: dirichlet problem}. According to Proposition \ref{prop: asymptotic Dirichlet}, as long as we can find a distance $R$ such that the uniform $C^0$ bound
\begin{equation} \label{eqn: uniform bound}
    d_Y(h_i(x),f(x)) \leq R
\end{equation}
holds for all $i$ and all $x \in \Omega_i$, the maps $h_i$ will sub-converge to a harmonic map $h$ from $X$ to $Y$, lying at distance at most $R$ from $f$. The key lemma will be the following interior $C^0$ estimate.
Recall that for $y \in Y$ we let $d_y \colon Y \to \R$ denote the distance function $d_y(y')=d_Y(y,y')$. Also recall that we define a harmonic measure on a domain $B$ with respect to a point $x$ to be any measure with the property that $\int_B u(z) \mu(z) \geq u(x)$ for any subharmonic function $u$. The assumption of the following Proposition could be rephrased as saying that $f$ is stable at scale $r$ at the point $x$.

\begin{proposition}[Interior $C^0$ estimate] \label{prop: interior estimate}
    Let $x\in X$. Suppose that $f\colon B(x,r+2) \to Y$ is $L$-coarsely Lipschitz and satisfies for each $y \in Y$,
    \begin{equation} \label{eqn: stability on balls}
        \int_{\overline{B(x,r)}} (d_y \circ f(z) - d_y \circ f(x)) \mu_x^y(z) \geq {1}
    \end{equation}
    for a measure $\mu_x^y$ that is harmonic with respect to $x$ and supported on the closed ball.
    
    There is a constant $R = R(K,n, r, L)$ such that if $h\colon B(x,r+2) \to Y$ is a harmonic map and $d_Y(h(z),f(z))$ is maximized at $z = x$, then
    \[
        d_Y(h(x),f(x)) \leq R.
    \]
\end{proposition}

According to this proposition, each map $h_i:\Omega_i \to Y$ will satisfy 
\[
d_Y(h(x),f(x)) \leq R
\]
provided that there is a point $x_i$ in $\Omega$ maximizing $d_Y(h_i(z),f(z))$ \emph{which lies at distance at least $r+2$ from the boundary}. Since the distance is at most $L$ on the boundary, it will be maximized at some point of $\Omega$; to deal with the possibility that this point lies close to $\partial \Omega_i$ we need to choose the exhaustion with a little bit of care. This will be done in the next section, following the proof of the interior estimate.

\subsection{Interior \texorpdfstring{$C^0$}{C0} estimate}

We begin with a variant of the gradient estimate \ref{prop: gradient estimate}, which gives a better estimate when the harmonic map sends a point very close to the boundary of the ball. We refer to it as a Harnack estimate.

\begin{lemma}[Harnack estimate] \label{lem: harnack estimate}
    Let $B(x,r+2) \subset X$, let $y \in Y$, and suppose that $h \colon B(x,r+2) \to Y$ is harmonic and satisfies
    \[
    d_y \circ h(x) = D \geq 1
    \]
    and for all $z \in B(x,r+2)$,
    \[
    d_y \circ h(z) \leq D + E 
    \]
    Then there is a constant $C(E,K,n,r)$ such that $\abs{\nabla h} \leq C \sqrt{D}$ on $B(x,r)$. 
\end{lemma}

\begin{proof}
    For $z \in B(x,r+2)$, let $d_x(z) = d_X(x,z)$. 
    For $\rho \leq r+2$, let $V(\rho) = \mathrm{Vol}(B(x,\rho))$. 
    Let $G_x(z) = G^B(x,z)$ be the Green's function of $B(x,r+2)$ centered at $x$. 

    Since $d_y \circ h$ is Lipschitz on $B(x,r+2)$, the integral
    \[
    \int_{B(x,r+2)} \Delta (d_y \circ h)^2(z) G_x(z) dz 
    \]
    makes sense weakly, and by Green's theorem
    \begin{align*}
    \int_{B(x,r+2)} \Delta (d_y \circ h)^2(z) G_x(z) dz = &\int_{\partial B(x,r+2)} (d_y \circ h)^2(z) \abs{\nabla G_x(z)} dz - (d_y \circ h)^2(x) \\
    \leq &(D + E)^2 - D^2 \\ 
    = &2ED + E^2.
    \end{align*}
    Here we have used that $d_y \circ h(z) \leq D + E$
    and that $\abs{\nabla G_x(z)} dz$ is a probability measure.

    Next we use the inequality $\Delta (d_y \circ h)^2 \geq 4 e$ of \ref{prop: K-S} part (2) to conclude
    \[
    4 \int_{B(x,r+2)} e(z) G_x(z) dz \leq 2ED + E^2
    \]
    By the Green's function lower bound \ref{thm: Greens function lower bound}, we can find $c$ such that for all $z \in B(x,r+1)$
    \[
    G_x(z) \geq \frac{c}{V(x,r+1)}
    \]
    and so since $e$ is non-negative,
    \[
    \frac{c}{V(x,r+1)} \int_{B(x,r+1)} e(z)dz \leq 2ED + E^2.
    \]
    Finally, since $\Delta e \geq -K e$ by Proposition \ref{prop: K-S} part (3), we can apply the mean value inequality \ref{thm: MVI} to conclude
    \[
    \sup_{B(x,r)} e(z) \leq C D.
    \]
    For some $C=C(n,K,r,E)$, using $D \geq 1$ to absorb the $E^2$ term.
    Since $e(z) = 2 \abs{\nabla h}^2$, we obtain the conclusion after taking square roots of both sides.
\end{proof}

The following lemma is a simple consequence of the Ptolemy inequality for CAT(0) spaces. It may be helpful to refer to Figure \ref{fig:thinking diagram} below.

\begin{lemma} \label{lem: cross ratio bound frfr}
    For any quadrilateral in a CAT(0) space with opposite side lengths $D,D'$, $E,E'$, and $F,F'$,
    \[
    F - D + F' - D' \leq \frac{2EE'}{D}
    \]
\end{lemma}

\begin{proof}
    The Ptolemy inequality \ref{prop: Ptolemy} gives $-DD' \leq - FF'+EE'$. Hence
    \begin{align*}
        D(F - D + F'  - D') &\leq DF  - D^2 + DF' - FF' + EE'\\
        &= (F-D)(D-F') + EE' \\
        &\leq 2EE'
    \end{align*}
    where the last line uses the triangle inequality twice.
\end{proof}

\begin{remark}
    When $Y$ is a manifold, the expression $F + F' - D - D'$ is the symplectic area of any quadrilateral in the cotangent bundle of $Y$ whose edges are the geodesics, $F,D,F',D'$ traversed in opposite directions at unit speed. It has a natural interpretation as the log of a cross-ratio \cite[Section 4.4]{labourie2007crossratios}. 
\end{remark}

We now proceed to the proof of the interior estimate.

\begin{proof}[Proof of Proposition \ref{prop: interior estimate}, interior estimate]
    For any $z \in \overline{B(x,r)}$, consider the quadrilateral in Figure \ref{fig:thinking diagram}. 
    By Lemma \ref{lem: cross ratio bound frfr} we have
    \[
    F - D +F' - D' \leq \frac{2EE'}{D}.
    \]
    Since the distance from $f$ to $h$ is maximized at $x$, we also have $D' \leq D$, whence
    \[
    F - D + F' - D \leq \frac{2EE'}{D}.
    \]

    Since $F' \leq D' + E \leq D + E$, and $E \leq L(r+3)$ since $f$ is coarse Lipschitz, we may apply the Harnack estimate, Lemma \ref{lem: harnack estimate}. We conclude that there is a constant $C$ depending on $n,K$, and $r$, and $L$ such that $|\nabla h| \leq C\sqrt{D}$ on $B(x,r)$. 
    Integrating from $x$ to $z$, we see that for all $z \in B(x,r)$, $E' \leq rC\sqrt{D}$.
    Hence, for such $z$, we have for a new constant $C$,
    \begin{equation} \label{eqn: cross ratio plus harnack}
        F - D + F' - D \leq \frac{C}{\sqrt{D}}.
    \end{equation}

     Now take $y=h(x)$ and consider the harmonic measure $\mu_x^y$ so that the stability inequality \eqref{eqn: stability on balls} is satisfied by $f$. 
     This gives
     \[
        \int_{\overline{B(x,r)}} (F(z) - D) \mu_x^y(z) \geq 1.
    \]
    On the other hand, since $d_{f(x)}$ is a convex function on $Y$ and $h$ is a harmonic map, $d_{f(x)} \circ h$ is subharmonic on $X$. 
    Since $\mu_x^y$ is harmonic with respect to the point $x$, integrating against $\mu_x^y$ gives
    \[
        \int_{\overline{B(x,r)}} (F'(z) - D) \mu_x^y(z) \geq 0.
    \]
    
    Noting that $\mu_x^y$ is a probability measure, and the right hand side of \eqref{eqn: cross ratio plus harnack} is independent of $z$, we conclude
    \[
    1 \leq \frac{C}{\sqrt{D}}
    \]
    in other words
    \[
        D = d_Y(h(x),f(x)) \leq C^2.
    \]

\end{proof}

\begin{figure}
\centering
\begin{tikzpicture}
    % Define coordinates
    \coordinate (fx) at (0,0);
    \coordinate (fz) at (1,0);
    \coordinate (hx) at (0.5,5);
    \coordinate (hz) at (3.5,5);

    % Unchanged lines
    \draw (fx) -- (hx) node[pos=0.7, left=2pt]  {$D$};
    \draw (fx) -- (hz) node[pos=0.7, left=2pt]  {$F'$};
    \draw (fz) -- (hx) node[pos=0.7, right=2pt] {$F$};
    \draw (fz) -- (hz) node[pos=0.7, right=2pt] {$D'$};
    \draw (fx) -- (fz) node[pos=0.5, below=1pt] {$E$};
    \draw (hx) -- (hz) node[pos=0.5, above=1pt] {$E'$};

    % Point labels
    \node[left=2pt] at (fx) {$f(x)$};
    \fill (fx) circle (1.5pt);

    \node[right=2pt] at (fz) {$f(z)$};
    \fill (fz) circle (1.5pt);

    \node[left=2pt] at (hx) {$h(x)$};
    \fill (hx) circle (1.5pt);

    \node[right=2pt] at (hz) {$h(z)$};
    \fill (hz) circle (1.5pt);

\end{tikzpicture}
\caption{Interior estimate}
\label{fig:thinking diagram}
\end{figure}
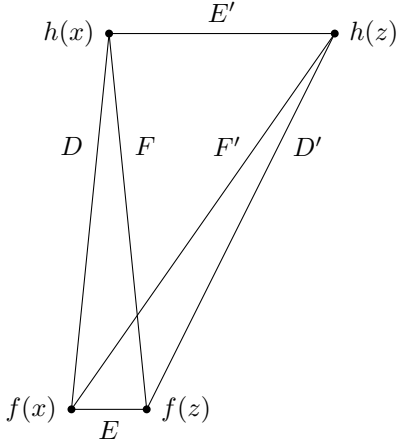

\subsection{Boundary \texorpdfstring{$C^0$}{C0} estimate}

In this section we construct an exhaustion $\Omega_i$ of $X$ for which we can control the distance between $f$ and the solutions $h_i$ to the Dirichlet problems \eqref{eqn: Dirichlet problem} in a fixed neighborhood of the boundary of $\Omega_i$.

\begin{definition}
    A domain $\Omega \subset X$ satisfies the \emph{exterior $\rho$-ball condition} if every point on $\partial\Omega$ lies on the boundary of a ball of radius $\rho$ contained in $X \setminus \Omega$.
\end{definition}

\begin{remark}
    When $X$ is a Hadamard manifold, any ball $B(x,r)$ satisfies the exterior $\rho$-ball condition for any $ \rho \ge 0$.
\end{remark}

It is easy to construct domains satisfying the exterior $\rho$-ball condition.

\begin{lemma} \label{lem: exterior 1 ball}
    If $\Omega$ is a domain in $X$, the domain $\Omega^{(\rho)}$ of points in $\Omega$ at distance at least $\rho$ from the boundary of $\Omega$ satisfies the exterior $\rho$-ball condition.
\end{lemma}

\begin{proof}
    Let $p$ be a point on $\partial \Omega^{(\rho)}$. Then there is a point $q \in \partial \Omega$ such that $d(p,q) = \rho$. The ball $B(q,\rho)$ lies outside of $\Omega^{(\rho)}$ and touches the boundary at $p$.
\end{proof}

The last ingredient in the proof of the main theorem is the following boundary estimate. 

\begin{prop}[Boundary $C^0$ estimate]\label{prop: boundary estimate} 
    Let $\Omega \subset X$ be a Lipschitz domain satisfying the exterior $1$-ball condition. 
    Suppose $f \colon \Omega \to Y$ is an $L$ coarse Lipschitz map, and $h \colon \Omega \to Y$ is a harmonic map whose distance from $f$ is bounded by $L$ on $\partial \Omega$.
    Let $x \in \Omega$ be any point at which $d_Y(f(z),h(z))$ is maximized. 
    For every $r > 0$, there is an $R(K,n,L,r)$, such that 
    \[
    d_X(x,\partial\Omega) \leq r \implies d_Y(f(x),h(x)) \leq R.
    \]    
\end{prop}

The proof will involve a lemma to the effect that if $\Omega$ satisfies the exterior ball condition then a Brownian motion that starts near the boundary has a positive probability of hitting the boundary before it travels too far. This is rephrased in the following lemma.
For the following lemma, refer to Figure \ref{fig: boundary estimate}.

\begin{figure}
\begin{center}
\begin{tikzpicture}[
    scale=.8, % Adjusts the overall size
    point/.style={fill, circle, inner sep=1.5pt} % Style for the points
  ]
    
    % --- Coordinates ---
    % Define the locations of the key points
    \coordinate (x) at (-1.5, 0.5);
    \coordinate (x1) at (0, 0);
    \coordinate (x2) at (1, -0.33);
    
    % --- Boundaries and Regions ---
    
    % S-shaped curve (right boundary of D)
    % We draw this in two parts, pivoting at (x1)
    \draw[thick] (1, 3.65) .. controls (-1.5, 2) and (0.5, 1) .. (x1);
    \draw[thick] (x1) .. controls (-0.5, -1) and (1.5, -2) .. (.5, -2.95);
    
    \draw (x2) circle (1.05);
  
    \draw[thick] (x) circle (4);	

    % --- Labels ---
    
    % Region D
    \node at (-3, 0) {\LARGE $D$};
    
    % Boundary \partial D labels
    \node at (-4.5, 2.5) {$\partial_2 D$};
    \node at (-.7, 2) {$\partial_1 D$};
    
    % --- Points and Connecting Lines ---
    
    % Line from x to x1
    \draw[thick] (x) -- (x1) node[midway, above, sloped] {$\leq r$};
    
    % Line from x1 to x2
    \draw[thick] (x1) -- (x2) node[midway, above, sloped] {$1$};
    
    % Points (drawn last to be on top)
    \node[point, label=above left:{$x$}] at (x) {};
    \node[point, label=below right:{$x_1$}] at (x1) {};
    \node[point, label=above right:{$x_2$}] at (x2) {};

\end{tikzpicture}
\end{center}
\label{fig: boundary estimate}
\caption{Boundary estimate}
\end{figure}
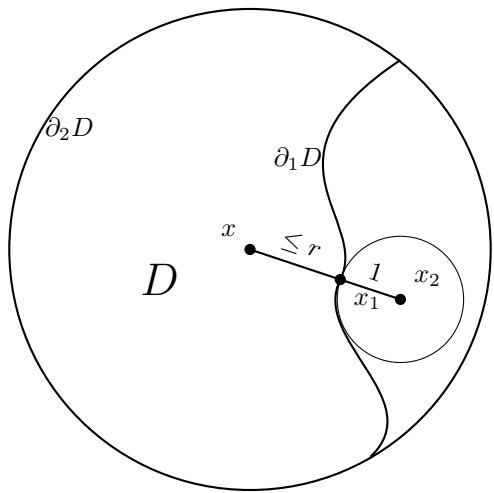

\begin{lemma}\label{lemma:exterior ball condition controls measure}
    Fix $r > 0$. Let $\Omega \subset X$ be a domain satisfying the exterior $1$-ball condition. Let $x \in \Omega$, and suppose $d_X(x,\partial \Omega)\leq r$. Let $D = B_{2r+3}(x) \cap \Omega$, and write
    \[
    \partial D = \partial_1 D + \partial_2 D = \partial \Omega \cap B_{2r+3}(x) + \Omega \cap \partial B_{2r+3}(x).
    \]
    Let $\mu_x$ be the hitting measure on $\partial D$ corresponding to the point $x \in D$. Then there is an $\alpha >0$ depending on $n$,$K$, and $r$ such that $\mu_x(\partial_1D) \geq \alpha$.
\end{lemma}

\begin{proof}
    The measure of $\partial_1(D)$ with respect to the hitting measure $\mu_x$ is equal to the value at $x$ of the function $u$ solving the Dirichlet problem
    \begin{equation} \label{eqn: hitting measure}
        \begin{cases}
            \Delta u = 0 \\
            u|_{\partial_1 D} = 1 \textrm{ a.e.} \\
            u|_{\partial_2 D} = 0 \textrm{ a.e.}
        \end{cases}
    \end{equation}
    We bound $u(x)$ from below using a subsolution to \ref{eqn: hitting measure}.

    Let $x_1$ be a point in $\partial \Omega$ such that $d_X(x,x_1) \leq r$, and let $x_2$ be a point in $X \setminus \Omega$ such that $d_X(x_1,x_2) = 1$ and $B_1(x_2)$ is disjoint from $\Omega$. For $z \in X$, let $\rho(z)=d(x_2, z)$.

    Let $\Gamma(\cdot)$ be the radial Green's function of the spaceform, given by equation \eqref{eqn: Gamma(r)}. 
    The Laplacian comparison theorem says that $\Gamma(\rho(z))$ is subharmonic on $X$. We will now renormalize this function to make it a subsolution to \eqref{eqn: hitting measure}. 
    
    For $z \in \overline{D}$, let
    \[
    v(z) = \frac{\Gamma(\rho(z)) - \Gamma(r+2)}{\Gamma(1) - \Gamma(r+2)}.
    \]
    Since $\Gamma(t)$ is strictly decreasing in $t$, and $x_2$ is distance at least $1$ from $\partial \Omega$, $v(z) \leq 1$ on $\partial \Omega$. 
    Also, since $d_X(x_2,x) \leq r + 1$, the triangle inequality gives $d_X(x_2, \partial B_{2r+3}(x)) \geq r + 2$, and so $v(z) \leq 0$ on $\partial B_{2r + 3}(x)$. We conclude that $v$ is indeed a subsolution.
    Therefore, with
    \[
    \alpha = \frac{\Gamma(r+1) - \Gamma(r+2)}{\Gamma(1) - \Gamma(r+2)} > 0
    \]
    we have
    \[
    u(x) \geq v(x) \geq \alpha.
    \]
\end{proof}

We now prove the boundary estimate.

\begin{proof}[Proof of Proposition \ref{prop: boundary estimate}]
    We suppose $d_X(x, \partial \Omega) \leq r$.
    Let $D = B_{2r+3}(x) \cap \Omega$, and let $\mu_x$ be the harmonic measure on $\partial D$ centered at $x$, and let $\alpha = \mu(\partial \Omega  \cap B_{2r+3}(x))$. 
    For $z \in \overline{D}$, let 
    \[
        u(z) = d_Y(f(x),h(z)).
    \]
    By the triangle inequality,
    \[
        u(z) \leq d_Y(f(z),h(z)) + d_Y(f(x),f(z)) \leq d_Y(f(z),h(z)) + (2r+3)L
    \]
    Since $u$ is subharmonic,
    \begin{align*}
        u(x) \leq& \int_{\partial D} u(z)\mu_x(z) \\
        \leq& \int_{\partial D} d_Y(f(z),h(z)) \mu_x(z) + (2r+4)L \\
        =& \int_{\partial_1 D} d_Y(f(z),h(z)) \mu_x(z) \\
        &+ \int_{\partial_2 D} d_Y(f(z),h(z)) \mu_x(z) + (2r+4)L\\
        \leq& L + u(x) (1-\alpha) + (2r+4)L
    \end{align*}
    Here we have used that $d_Y(f(z),h(z)) \leq L$ on $\partial_1 D$ and that $d_Y(f(z),h(z)) \leq u(x)$ on $\partial_2 D$.
    
    Therefore, 
    \[
    u(x)=d_Y(f(x),h(x)) \leq \frac{(2r+5)L}{\alpha}.
    \]

\end{proof}

We note that stability of $f$ does not play a role in the boundary estimate.

\subsection{Conclusion of the proof}

We have now completed all the steps of the proof of the main existence theorem \ref{thm: main Donaldson-Corlette}, which we now restate in a more quantitative form:

\begin{theorem} \label{thm: Donalson-Corlette}
    Let $X$ be a complete $n$-manifold with Ricci curvature bounded below by $-K$, and $Y$ a proper Hadamard space. 
    Let $f\colon X \to Y$ be an $L$-coarsely Lipschitz map which is stable at scale $r > 0$. 
    Then there is a constant $R(n,K,L,r)$ such that there exists a harmonic map $h: X \to Y$ whose distance to $f$ is at most $R$.
\end{theorem}

We have already explained every element of the proof, but for the reader's convenience we put everything together in the following.

\begin{proof}
Let $\Omega_i$ be an exhaustion of $X$, and let $\Omega_i^{(1)}$ be the exhaustion constructed in Lemma \ref{lem: exterior 1 ball}, which satisfies the exterior 1-ball condition. Let $f^{(1)}$ be the mollification of $f$ constructed in Proposition \ref{prop: Lipchitz}. For each $i$, let $h_i$ be the solution to the Dirichlet problem on $\Omega_i^{(1)}$ which is equal to $f^{(1)}$ on the boundary. For each $i$, let $x_i$ be a point in $\Omega_i^{(1)}$ at which the distance between $h_i$ and $f^{(1)}$ is maximized.

Now let $r$ be such that $f$ is stable at scale $r$. If $d_Y(x_i, \partial \Omega_i^{(1)}) \leq r+2$, then the boundary estimate Proposition \ref{prop: boundary estimate} gives a constant $R$ such that $d_Y(f(x_i),h_i(x_i)) \leq R$. 
On the other hand, if $d_Y(x_i, \partial \Omega_i^{(1)}) \geq r+2$ then $B(x_i, r+2) \subset \Omega_i^{(1)}$, and the interior estimate Proposition \ref{prop: interior estimate} gives a constant $R'$ such that $d_Y(f(x_i),h_i(x_i)) \leq R'$. 
Hence for any $z \in \Omega_i^{(1)}$, 
\[
d_Y(f(z),h_i(z)) \leq d_Y(f(x_i),h_i(x_i)) \leq \mathrm{max}(R,R').
\]
This provides the a priori $C^0$ bound necessary to conclude the existence of a limiting $h$ from Proposition \ref{prop: asymptotic Dirichlet}.
\end{proof}

\begin{remark}
    It would be interesting to reformulate our argument using harmonic map heat flow instead of an asymptotic Dirichlet problem.
\end{remark}

\section{Uniqueness and equivariant maps} \label{sec: uniqueness}

In this section, we prove Theorems \ref{thm: main uniqueness} and \ref{thm: main equivariant}. 
We assume throughout this section that $Y$ is a manifold, and for the theorems, we assume further that $Y$ is a symmetric space. 
Some more generality is possible, but some version of the Jacobi field extension property of Proposition \ref{prop: Y has condition E} for $Y$ is a necessary ingredient. 

We also assume that $X$ is connected and that the isometry group of $X$ acts cocompactly on $X$.
We believe that cocompactness is not an essential assumption, and could be removed by using the theory of metric-measure spaces with lower Ricci curvature bounds, or more quantitative estimates on $Y$.

\subsection{Variations through geodesics}

Let $I$ be an interval of $\R$. We call $H: I \times \Omega \to Y$ a variation through geodesics if $t \mapsto H(t,x)$ is a geodesic for every $x \in \Omega$. If $H$ is a variation through geodesics, the vector fields obtained by differentiating $H$ in directions tangent to $\Omega$ are Jacobi fields. We call $H$ a parallel variation through geodesics if every such Jacobi field is parallel along the geodesic. Note that a Jacobi field is parallel if and only if its norm is constant in $t$. Furthermore, since $Y$ is a Hadamard manifold, the norm of a Jacobi field is a convex function of $t$.

Parallel variations through geodesics arise in the following two situations.

\begin{prop} \label{prop: constant distance}
    Let $h_i\colon \Omega \to Y$, $i = 0,1$ be harmonic maps at constant distance, and let $H\colon [0,1] \times \Omega \to Y$ be the geodesic interpolation between them. 
    Then $H$ is a parallel family of geodesics.
\end{prop}

\begin{prop} \label{prop: harmonic busemann function}
    Let $h\colon \Omega \to Y$ be harmonic, and let $b$ be a Busemann function on $Y$. 
    Suppose $b \circ h$ is harmonic on $\Omega$. 
    Let $H\colon [0, \infty) \times \Omega \to Y$ be the variation through geodesics obtained by flowing $h$ by the negative gradient flow of $b$.
    Then $H$ is a parallel family of geodesics. 
\end{prop}

Both propositions are well-known and follow from the first and second variation formulas for the induced metric along the family. For the first, one uses the negative gradient flow of $\tfrac{1}{2}d^2$ on $Y \times Y$.

\begin{definition}
    We say $Y$ \emph{has extendable parallel Jacobi fields} if every Jacobi field on a geodesic $\gamma\colon \R \to Y$ that is parallel along an interval of $\R$ is parallel along the whole geodesic.
\end{definition}

This is a natural extension of Labourie's condition `sans demies-bandes plates' in \cite{Labourie1991existenceharmoniques}. 

Not every manifold has extendable parallel Jacobi fields; for instance, a non-flat surface with a flat open sub-surface will not. However, since the coefficients of the Jacobi equation of a real analytic manifold are real analytic, we have for instance:

\begin{prop} \label{prop: Y has condition E}
    A real analytic manifold has extendable parallel Jacobi fields.
\end{prop}

\subsection{Uniqueness}

\begin{theorem}\label{thm: uniqueness}
    Suppose that $X$ is connected with cocompact isometry group and that $Y$ is a symmetric space. 
    Every stable coarsely Lipschitz map from $X$ to $Y$ is at bounded distance from a \emph{unique} harmonic map.
\end{theorem}

\begin{proof}
    We only need to show uniqueness, since existence is Theorem \ref{thm: main Donaldson-Corlette}.
    
    Let $h$ and $h'$ be two coarse Lipschitz harmonic maps from $X$ to $Y$ at bounded distance from one another, and suppose $h$ is stable at scale $r$. Let $D = \sup_{x \in X} d_Y(h(x), h'(x))$, and suppose for the sake of contradiction that $D > 0$. Let $x_i$ be a sequence of points in $X$ such that $d_Y(h(x_i),h'(x_i))$ approaches $D$. 

    Passing to a subsequence of the $x_i$, use the cocompactness assumption on $X$ to choose a sequence $\gamma_i \in \mathrm{Isom}(X)$ that $ \gamma_i \cdot x_i$ converges to a point $x \in X$. Fix a point $y \in Y$ and choose $g_i \in \mathrm{Isom}(Y)$ such that $g_i \circ h \circ \gamma_i^{-1} (x) = y$.
    
    Let $h_i = g_i \circ h \circ \gamma_i^{-1}$ and $h'_i = g_i \circ h'\circ \gamma_i^{-1}$. Since $h_i$ and $h'_i$ each send $x$ to a $D$-neighborhood of $y$ and are uniformly coarse Lipschitz, by Proposition \ref{prop: asymptotic Dirichlet} they subsequentially converge to harmonic maps $h_\infty$ and $h'_\infty$ from $X$ to $Y$. Since the convergence is locally uniform, $h_\infty$ is still stable at scale $r$. Moreover they remain at distance at most $D$, except that now $d_Y(h_\infty(x), h'_\infty(x)) = D$, so that the maximum is attained. Since the distance between $h_\infty$ and $h'_\infty$ is a subharmonic function on $X$, the strong maximum principle implies that the distance is constant.

    Hence by Propositions \ref{prop: constant distance} and \ref{prop: Y has condition E}, there is a parallel family $h_t, t \in \R$ of harmonic maps from $X$ to $Y$ such that $h_0 = h_\infty$ and $h_D = h'_\infty$. Let $\eta^+$ be the forward endpoint of the geodesic ray $t \mapsto h_t(x)$, and let $\eta^-$ be the backward endpoint. Since the family is parallel, the distance between the geodesics $t \mapsto h_t(x)$ and $t \mapsto h_t(z)$ is uniformly bounded for every $z$ in $X$, so $\eta^{\pm}$ are also the endpoints of each geodesic $t \mapsto h_t(z)$. The function $b_{\eta^+} + b_{\eta^-}$ is constant on the whole parallel set determined by $\eta^+$ and $\eta^-$. Therefore the function $b_{\eta^+}\circ h_\infty(z) + b_{\eta^-} \circ h_\infty(z)$ is constant in $z$. So $b_{\eta^+}\circ h_\infty(z)$ is superharmonic as well as subharmonic, so it is harmonic, and therefore
    \[
    \int_X b_{\eta^+} \circ h_\infty(z) - b_{\eta^+} \circ h_\infty(x) \mu_x(z) = 0
    \]
    for every harmonic measure $\mu_x$. This contradicts the stability of $h_\infty$ at scale $r$, so we must have $D=0$ and $h=h'$.
\end{proof}

\subsection{Stability of equivariant maps}

In this subsection we prove Theorem \ref{thm: main equivariant}. 
\begin{theorem} \label{thm: equivariant}
    Suppose that $G$ is a semisimple Lie group and $Y$ is the associated symmetric space of non-compact type. 
    Suppose $X$ is connected and $\Gamma$ acts cocompactly by isometries on $X$.
    If $f\colon X \to Y$ is equivariant under a Zariski-dense action $\rho$ of $\Gamma$ on $Y$, then $f$ is stable.
\end{theorem}

\begin{proof}
    Since $\rho$ is Zariski-dense, its Zariski closure is reductive. Hence by the Donaldson-Corlette theorem, there exists an equivariant harmonic map $h$, necessarily at bounded distance from $f$. Since stability is coarse, it suffices to show that $h$ is stable.

    Suppose not. For a point $x \in X$ and $r>0$, let $\mu_x^r$ be the hitting measure on $\partial B(x,r)$, and let $P_r$ be the corresponding convolution operator. First, we claim that for any $k$ and $r$, there is a point $x\in X$ and a Busemann function $\eta \in \partial_\infty Y$ such that
    \begin{equation} \label{eqn: Pk}
        P_r(b_{\eta} \circ h)(x) - b_{\eta} \circ h(x) < 1/k.
    \end{equation}
    For if this were not true, then using Proposition \ref{prop: Busemann enough} to restrict to Busemann functions, and telescoping as in Proposition \ref{prop: subharmonicity is coarse}, $h$ would be stable at scale $kr$ with respect to the operator $P_r^k$.

    So choose a sequence $r_k \to \infty$, and a sequence $x_k$ and $\eta_k$ such that \eqref{eqn: Pk} holds. Using cocompactness of the action on $X$ and compactness of $\partial_\infty Y$, we may assume up to replacing $(x_k, \eta_k)$ by $(\gamma_k^{-1} x_k, \rho(\gamma_k)\eta_k)$ and passing to a subsequence that $x_k$ converges to $x \in X$ and $\eta_k$ converges to $\eta \in \partial_\infty Y$.

    Since $b_{\eta_k} \circ h$ is subharmonic, we have 
    \[
    P_r(b_{\eta_k} \circ h) \leq P_s(b_{\eta_k} \circ h)
    \]
    for any $r \leq s$. Since $b_{\eta_k}$ converges locally uniformly to $b_\eta$, we can take the limit in \eqref{eqn: Pk}, holding $r$ fixed to conclude that for any $r$,
    \begin{equation} \label{eqn: harmonic}
        P_r (b_{\eta} \circ h)(x) - b_{\eta} \circ h(x) = 0. 
    \end{equation}
    We conclude that $b_\eta \circ h$ is harmonic on the ball of radius $r$ around $x$ for any $r$.
  
    For any $s \in \R$, let $h_s$ be the map at signed distance $s$ from $h$ along the geodesic towards $\eta$. It follows from proposition \ref{prop: Busemann enough} that $H(s,x) = h_s(x)$ is a parallel family of geodesics for $s \geq 0$, and then from proposition \ref{prop: Y has condition E} that $H$ is parallel for all $s \in \R$. Hence, the point $\eta'$ opposite to $\eta$ through $h(x)$ is independent of $x$, and each point $h(x)$ lies on a geodesic from $\eta'$ to $\eta$. But the locus of points $P_{\eta,\eta'}$ with this property, i.e.\ the parallel set of $\eta$ and $\eta'$, is a proper real algebraic subset of $Y$. Hence for any $x_0 \in X$, the condition that $\gamma \cdot h(x_0) \in P$ is a nontrivial real algebraic condition on the group $\Gamma$. We see that $\Gamma$ is not Zariski dense, contrary to assumption. Hence $h$, and likewise $f$, must be stable.
\end{proof}

\begin{remark}\label{rmk: LK}
    If $\Gamma$ acts cocompactly on $X$ and $\rho$ is an action of $\Gamma$ on $Y$, one constructs a locally bounded equivariant map as follows: choose a compact set $K \subset X$ such that $\Gamma \cdot K = X$, and a point $y$ of $Y$, and let $f(x)$ be the circumcenter in $Y$ of the compact set 
    \[
    \{\rho(\gamma) y | \gamma \in \Gamma \textrm{ and } x \in \gamma K\}.
    \]
    Since any two locally bounded equivariant maps are at bounded distance, and stability is a coarse notion, one could ask whether stability of the representation $\rho$ in the sense of Donaldson-Corlette is the same as stability in our sense of any $\rho$-equivariant map.

    The answer is that, while the two notions are related, stability of the equivariant map is stronger. What we have shown above is that the equivariant map is stable in our sense exactly when the harmonic map is unique among harmonic maps at bounded distance; on the other hand the representation $\rho$ is stable in the sense of Donaldson-Corlette exactly when the harmonic map is unique among \emph{equivariant} harmonic maps.
 
    The difference, essentially the only difference up to factoring products, comes from stable representations with compact image. For such representations \emph{any} constant map is a harmonic map at bounded distance, so they cannot be stable in our sense.
\end{remark}

\section{\texorpdfstring{$\Theta$}{Theta}-quasi-isometric embeddings and stability}

In this section we prove a more general version of the criterion for stability in terms of boundary maps, Theorem \ref{thm: intro criterion for stability}. 
We first give the definition of $\Theta$-quasi-isometric embeddings from a metric space to a symmetric space $Y$, and recall the higher rank Morse lemma of \cite{kapovich2018morse}, and the existence of continuous extensions to the flag variety. 
We then prove a slightly stronger version of the extension theorem than we could find in \cite{kapovich2018morse} (Proposition \ref{prop: uniform convergence}), and show that stability of $\Theta$-quasi-isometric embeddings can be deduced as a consequence of stability in the sense of \cite{KapovichLeebMillson2009ConvexFunctionsSymmetric} of certain measures on the visual boundary of the symmetric space. 
In the case that $Y=Y_d$ is the symmetric space of $\PGL_d(\R)$, $\Theta = \Pi$, the set of all simple roots, and $f$ has positive boundary map (Definition \ref{def: positive}), this will allow us to conclude that $f$ is stable.

\subsection{Symmetric space reminders}

Let $G$ be a connected semisimple Lie group with finite center and let $Y$ be the associated symmetric space of non-compact type.
There is a unique $G$-invariant Riemannian metric $g$ on $Y$ up to rescaling on each de Rham factor.
It will be convenient to fix the metric so that the minimium sectional curvature of any tangent $2$-plane is $-1$ in each de Rham factor. 
Fix a point $o\in Y$, and let $K$ be the stabilizer of $o$ in $G$. The Lie algebra $\mathfrak{g}$ of $G$ admits a \emph{Cartan involution} $\theta$ whose fixed point set is precisely $\mathfrak{k}$, the Lie algebra of $K$. 
Let $\mathfrak{p}$ denote the $-1$-eigenspace of $\theta$. For any point $y \in Y$ there exists a unique $v \in \mathfrak{p}$ so that $e^v o = y$.

Let $\mathfrak{a}$ denote a Cartan subspace of $\mathfrak{p}$.
The adjoint action of $\mathfrak{a}$ on $\mathfrak{g}$ induces the restricted root space decomposition of $\mathfrak{g}$ and we let $\Sigma \subset \mathfrak{a}$ denote the set of restricted roots. 
Choose a subset $\Pi \subset \Sigma$ of simple roots and a corresponding positive closed Weyl chamber $\mathfrak{a}^+ \subset \mathfrak{a}$. 
For any $v \in \mathfrak{p}$ there exists a unique $a \in \mathfrak{a}^+$ and a (typically non-unique) element $k \in K$ so that $\Ad(k)(v)=a$. 

In particular, for any pair of points $y_1, y_2$ in $Y$, there exists a unique element $a \in \mathfrak{a}^+$ and an element $g \in G$ so that 
    \[ g(y_1,y_2) = (o,e^a o) ,\]
the element $a \in \mathfrak{a}^+$ is called the \emph{vector-valued distance} of $(y_1,y_2)$.
%{warning: this fact uses that $G$ is connected!} 
For a pair of points $y_1,y_2 \in Y$ we will often write $\vec{d}_Y(y_1,y_2)$ for the vector-valued distance.

\subsection{\texorpdfstring{$\Theta$}{Theta}-quasi-isometric embeddings}\label{sec: urqies}

\begin{definition}\label{def: urqie}
    Let $Z$ be a metric space, and let $\Theta$ be a non-empty subset of $\Pi$. 
    A map $f \colon Z \to Y$ is called an \emph{$(L,M,\Theta)$-quasi-isometric embedding} if it is $L$-coarsely Lipschitz and 
        \[ \forall \alpha \in \Theta, \quad \forall z_1,z_2 \in Z, \quad \frac1{M} d_Z(z_1,z_2)  - 1 \le (\alpha \circ \vec{d}_Y)(f(z_1),f(z_2)) .\]
    A map $f \colon Z \to Y$ is called a \emph{$(M,\Theta)$-quasi-isometric embedding} if it is an $(L,M,\Theta)$-quasi-isometric embedding for some $L$, and a \emph{$\Theta$-quasi-isometric embedding} if it an $(M,\Theta)$-quasi-isometric embedding for some $M$.
\end{definition}

$\Theta$-quasi-isometric embeddings are equivalent to $\tau_{mod}$-uniformly regular quasi-isometric embeddings as introduced by Kapovich-Leeb-Porti \cite{kapovich2018morse}. 
They are also strongly related to the uniformly hyperbolic bundles of \cite{bridgeman2025ghostpolygonspoissonbracket}.
When $Y$ is negatively curved (i.e.\ has rank $1$), there is a unique simple restricted root $\Pi=\{\alpha\}$. 
In this case, $\{\alpha\}$-quasi-isometric embeddings are the same as quasi-isometric embeddings. 
In higher rank, every $\Pi$-quasi-isometric embedding is a quasi-isometric embedding but the converse does not hold in general.

Note that the vector-valued distance is not symmetric in general, but instead satisfies $-w_0 \vec{d}_Y(y_1,y_2) = \vec{d}_Y(y_2,y_1)$ where $w_0$ is the longest element of the Weyl group.
The isometry $-w_0 \colon \mathfrak{a} \to \mathfrak{a}$ is called the opposition involution, and it is non-trivial exactly when the root system of $\mathfrak{a}$ is of type $A_n,D_{2n+1}$ or $E_6$.
In any case, interchanging the roles of $z_1$ and $z_2$ in Definition \ref{def: urqie} shows that $(M,\Theta)$-quasi-isometric embeddings are also $(M,\Theta \cup (-w_0 \Theta))$-quasi-isometric embeddings.
So we may, and do from now on, assume that $\Theta$ is closed under the opposition involution.

\subsubsection{Flag topology}

Let $\mathcal{F}_\Theta$ denote the flag manifold associated to $\Theta$.
We need to recall the definition of the \emph{flag topology} on $Y \cup \mathcal{F}_\Theta$, making it into a suitable bordification of $Y$. 

There are several possible descriptions of the flag topology. 
We will give a description as a quotient of a subspace of the visual compactification. 
It can also be described in terms of a Satake compactification or a suitable horofunction compactifications \cite{guivarc2012compactifications}.

The visual boundary of $Y$ has a simplicial structure as a spherical building, inherited from the Coxeter complex structure on the visual boundary of each maximal flat. 
In particular, there is a type map from $\partial_\infty Y$ to the model spherical Weyl chamber $\sigma_{mod} = \mathbb{P}(\mathfrak{a}^+)$. 
The fibers of the type map are precisely the $G$-orbits of $\partial_{\infty} Y$. 
There is a distinguished face $\tau_\Theta$ of $\sigma_{mod}$ corresponding to $\Theta$: it is the intersection of the walls $\ker \alpha$ for $\alpha \in \Pi \setminus \Theta$ with $\sigma_{mod}$. 
The flag manifold $\mathcal{F}_\Theta$ is naturally identified with the set of simplices of $\partial_\infty Y$ of type $\tau_\Theta$.

The $\Theta$-regular part of $\sigma_{mod}$ is the subset where $\alpha$ is positive for each $\alpha \in \Theta$. 
The $\Theta$-regular part of $\partial_\infty Y$, denoted $\partial^{\Theta-reg}_\infty Y$, is the subset of $\partial_\infty Y$ with types in the $\Theta$-regular part of $\sigma_{mod}$.

Fixing $o \in Y$, let 
    \[ o^{\Theta-sing} = \{ y \in Y :  \exists \alpha \in \Theta \text{ s.t. } \alpha \circ \vec{d}_Y(o,y) =0 \} .\]
There is a natural continuous projection 
    \[ \overline{\pi_\Theta} = \pi_\Theta \cup \partial \pi_\Theta \colon (Y \setminus o^{\Theta-sing}) \cup \partial_\infty^{\Theta-reg}Y \to \mathcal{F}_\Theta \]
taking $y \in Y \setminus o^{\Theta-sing}$ to the unique simplex of type $\tau_\Theta$ in any chamber containing $\vec{oy}$. 
If a segment $y_1y_2$ is congruent to $oy$ for $y \in Y\setminus o^{\Theta-sing}$ then it is called \emph{$\Theta$-regular}.
For $\tau \in \mathcal{F}_\Theta$ (respectively $y \in Y \setminus o^\mathrm{\Theta-sing}$), the fiber $\pi_\Theta^{-1}(\tau)$ is called a \emph{Weyl cone}, and denoted $V_\Theta(o,\tau)$ (respectively, $V_\Theta(o,y)$).

Then the flag topology on $Y \cup \mathcal{F}_\Theta$ is the quotient topology under the map 
\[
 \mathrm{Id} \cup \partial \pi_\Theta \colon Y \cup \partial_\infty^{\Theta-reg}Y \to Y \cup \mathcal{F}_\Theta
\] 
which is independent of the basepoint $o \in Y$.

\subsubsection{The higher rank Morse lemma}

As a consequence of Gromov hyperbolicity, quasigeodesics in rank 1 symmetric spaces enjoy the Morse lemma: they lie uniformly close to geodesics.  
Kapovich-Leeb-Porti have proven a generalization of this statement for $\Theta$-quasigeodesics into symmetric spaces of higher rank.

\begin{theorem}[{The higher rank Morse lemma \cite[Theorem 1.3]{kapovich2018morse}, see also \cite[Theorem 7.3]{bochi2019anosov}}]\label{thm: higher rank morse lemma}
    For every $L,M$ there exists $D=D(L,M)$ so that:

    Every $(L,M,\Theta)$-quasigeodesic $q \colon [a,b] \to Y$ is at distance at most $D$ from the Weyl cone $V_\Theta(q(a),q(b))$, as long as $\vec{d}_Y(q(a),q(b))$ is $\Theta$-regular, which holds once $b-a > M$.     
\end{theorem}

When $\Theta = \Pi$, a Weyl cone $V_\Pi(y,\sigma)$ is an open Euclidean Weyl chamber with tip at $y$ asymptotic to the ideal Weyl chamber $\sigma \subset \partial_\infty Y$.
That case already covers the material we need in Section \ref{sec: universal hitchin} and for Theorem \ref{thm: main schoen}, and we encourage the reader to keep that special case in mind. 

Kapovich-Leeb-Porti also obtain the following consequence of Theorem \ref{thm: higher rank morse lemma}.
A pair $\tau,\tau' \in \mathcal{F}_\Theta$ is called \emph{transverse} if there exists a geodesic $c \colon \R \to Y$ so that $c(+\infty)$ is in the relative interior of $\tau$ and $c(-\infty)$ is in the relative interior of $\tau'$.

\begin{theorem}[{Boundary maps \cite[Theorem 1.4]{kapovich2018morse}}]\label{thm: boundary maps}
    Let $Z$ be a locally compact geodesic metric space and suppose that $q \colon Z \to Y$ is an $(L,M,\Theta)$-quasi-isometric embedding. 
    Then $Z$ is Gromov hyperbolic and the map $q$ extends continuously to a map 
        \[ \overline{q} = q \cup \partial q \colon Z \cup \partial_\infty Z \to Y \cup \mathcal{F}_\Theta \]
    taking distinct pairs of $\partial_\infty Z$ to transverse pairs in $\mathcal{F}_\Theta$.
\end{theorem}

Our next goal is to prove Proposition \ref{prop: uniform convergence}, which states that $L$-Lipschitz $(M,\Theta)$-quasi-isometric embeddings taking some point $x$ to some point $y$ which converge converge locally uniformly on $Z$ actually converge uniformly as maps $Z \cup \partial_\infty Z$ to $Y \cup \mathcal{F}_\Theta$. 
This is more-or-less a direct consequence of the higher rank Morse lemma, but we could not find the precise result we need in the literature.

\medskip

Recall that $\mathfrak{a}^+$ is a closed Euclidean Weyl chamber.
Define $m_\Theta \colon \mathfrak{a}^+ \to \R^{\geq 0}$ by
    \[
        m_\Theta(v) \coloneqq \min_{\alpha \in \Theta} \alpha(v),
    \]
and for brevity write also for $y \in Y$,
\[
    m_\Theta(y) = m_\Theta(\vec{d}_Y(o,y)).
\]
The flag variety $\mathcal{F}_\Theta$ inherits a Riemannian metric $g_\mathcal{F}$ as a quotient of $K$, the stabilizer of $o$ in $G$.
Let $d_\mathcal{F}$ be the corresponding distance function.
Using the projection $\pi_\Theta$ from $Y \setminus o^{\Theta-\mathrm{sing}}$ to $\mathcal{F}_\Theta$, we can also evaluate $d_\mathcal{F}$ on points of $Y \setminus o^{\Theta-\mathrm{sing}}$.

\begin{lemma} \label{lem: dF}
    Let $D=D(L,M)$ be the constant from the higher rank Morse lemma (Theorem \ref{thm: higher rank morse lemma}).
    Let $q \colon [0,r'] \to Y$ be an $(L,M,\Theta)$-quasigeodesic in $Y$ with $q(0) = o$, and suppose $M(D+1) < r \leq r'$. 
    Then $\vec{oq(r)}$ and $\vec{oq(r')}$ are $\Theta$-regular and
    \begin{equation} \label{eqn: contraction}
         d_\mathcal{F}(q(r),q(r')) \leq \frac{D}{\sinh(m_\Theta(q(r)) - D)}
    \end{equation}
\end{lemma}

\begin{proof}
    The inequality $M(D + 1) < r$ implies by the $(M,\Theta)$-quasi-isometric embedding property that $m_\Theta(q(r)) > D$. 
    Since $m_\Theta(q(r))$ and $m_\Theta(q(r')) > 0$, both $\vec{oq(r)}$ and $\vec{oq(r')}$ are $\Theta$-regular.
    
    Let $g_Y$ denote the Riemannian metric of $Y$. 
    The fiber bundle $\pi_\Theta \colon Y \setminus o^{\Theta-\mathrm{sing}} \to \mathcal{F}_\Theta$
    is $K$ equivariant with fiber the Weyl cone $V(o,\tau)$ over $\tau \in \mathcal{F}_\Theta$.
    Letting $g_V$ denote the restriction of $g_Y$ to $V(o,\tau)$ we have 
    \begin{equation} \label{eqn: g > gF}
           g_Y \geq g_V + \sinh(m(v))^2 g_\mathcal{F}
    \end{equation}
    which follows from considering the orbit map $K \times V(o,\tau) \to Y \setminus o^{\Theta-sing}$ given by $(k,v) \mapsto kv$. 
    More precisely, if $\mathfrak{k}_\tau$ denotes the infinitesimtal stabilizer of $\tau$ in $\mathfrak{k}=Lie(K)$, consider the fundamental vector fields induced by its perp $\mathfrak{k}^\tau$ (with respect to the restriction of the Killing form of $\mathfrak{g}=Lie(G)$ to $\mathfrak{k}$) at both $y \in Y$ and $\tau \in \mathcal{F}_\Theta$. 
    A key point is that these Killing vector fields are perpendicular to $V(o,\tau)$, so the decomposition above is orthogonal. 
    
    If $q$ is an $(L,M,\Theta)$-quasigeodesic from $y$, and $r \leq r'$, then by the higher rank Morse lemma (Theorem \ref{thm: higher rank morse lemma}), $q(r)$ is at most distance $D$ from $V(o,q(r'))$. 
    Let $\gamma$ be a minimizing path from $q(r)$ to $V(o,q(r'))$. 
    Since the length of the projection of $\gamma$ to $V$ is at most $D$, the lower bound $m(\gamma(t)) \geq m_\Theta(q(r)) - D > 0$ on all of $\gamma$ follows from the vector-valued triangle inequality:
        \[ \alpha(\vec{d}_Y(o,q(r)) - \vec{d}_Y(o,\gamma(t))) \le d(q(r),\gamma(t)) \le D \]
    see for instance \cite[Corollary 3.11]{Riestenberg2025quantified}.
    Since the $g$-length of the projection of $\gamma$ to $\mathcal{F}_\Theta$ is at most $D$, it follows from \eqref{eqn: g > gF} that the distance in $\mathcal{F}_\Theta$ from $q(r)$ to $q(r')$ is at most
    \[
        \frac{D}{\sinh(m_\Theta(q(r)) - D)}.
    \]
\end{proof}

\begin{lemma} \label{lem: triangle inequality}
    Suppose $f_1$ and $f_2$ are $(L,M,\Theta)$-quasi-isometric embeddings from $Z$ to $Y$ with $f_i(x) = o$ and such that for all $z \in B(x,r)$, $d_Y(f_1(z),f_2(z)) \leq E$. 
    Then for all $z' \in Z \setminus B(x,M(D+1))$, $f_1(z')$ and $f_2(z')$ are $\Theta$-regular and
    \begin{equation} \label{eqn: triangle inequality}
        d_\mathcal{F}(f_1(z'), f_2(z')) \leq \frac{2D}{\sinh(m_\Theta(q(r)) - D)} + E.
    \end{equation}
\end{lemma}

\begin{proof}
    This follows from the triangle inequality. We let $z$ be the corresponding point on the boundary of the ball of radius $r$, and we bound
    \[
    d_\mathcal{F}(f_1(z'), f_2(z')) \leq d_\mathcal{F}(f_1(z'), f_1(z)) + d_\mathcal{F}(f_1(z), f_2(z)) + d_\mathcal{F}(f_2(z), f_2(z'))
    \]
    and by bounding the first and third terms by Lemma \ref{lem: dF}, we obtain \eqref{eqn: triangle inequality}.
\end{proof}

\begin{proposition} \label{prop: uniform convergence}
    Suppose $f_n$ are a sequence of $(L,M,\Theta)$-quasi-isometric embeddings that converge locally uniformly to $f$ on $Z$. 
    Then their extensions $f_n \cup \partial f_n$ converge uniformly on $Z\cup \partial_\infty Z$.
\end{proposition}

\begin{proof}
In view of (\ref{eqn: g > gF}), the maps $f_n \cup \partial f_n \colon Z \cup \partial_\infty Z \to Y \cup \mathcal{F}_\Theta$ converge uniformly on any compact subset of $Z$.

It remains to show that the convergence is still uniform on a neighborhood of $\partial_\infty Z$. 
We first prove uniform convergence of $\pi_\Theta \circ f_n$.
For points close to $\partial_\infty Z$, but not on it, this follows from Lemma \ref{lem: triangle inequality}, since we can make both terms on the right of \eqref{eqn: triangle inequality} arbitrarily small.
Let $c$ be a geodesic ray in $Z$ emanating from $x$ and let $r_m \to \infty$. 
The sequence $\pi_\Theta \circ f \circ c(r_m)$ is Cauchy by Lemma \ref{lem: dF}, and converges to $\partial f(c)$. 
To see that the convergence $\partial f_n(c) \to \partial f(c)$ is uniform with respect to $d_\mathcal{F}$, take the limit in (\ref{eqn: triangle inequality}) as $z' \to \infty$. 

Now we consider convergence in the fibers of $\pi_\Theta$.
For $0 < \epsilon < \frac1{ML}$, once $r_0$ is sufficiently large according to an easy explicit bound in terms of $L,M,\epsilon$, all the compositions $f_n\circ c$ on $[r_0,\infty)$ take values in the subset 
    \[ Y'(o,\epsilon,\Theta)=\left\{y \in Y \setminus \{o\} : \forall \alpha \in \Theta, \frac{ \alpha \circ \vec{d}_Y(o,y)}{d(o,y)} \ge \epsilon \right\}  \subset Y \setminus o^{\Theta-sing}\]
forming angle at least $\epsilon$ from the walls corresponding to $\Theta$, as seen from $o$. 
The closures of $Y'(o,\epsilon,\Theta) \cap (Y \setminus B(o,r)) \cap V(o,\tau)$ in the flag topology of $Y \cup \mathcal{F}_\Theta$ converge uniformly in $r$ to the singleton $\{\tau\}$.
In particular, for $r$ large enough, $f_n \circ c (r)$ is uniformly close to $\partial f_n(c)$ in the flag topology, and hence to $\partial f(c)$ and to $f \circ c(r)$ by the previous paragraph.
\end{proof}

\begin{corollary} \label{cor: compactness of urqies}
    The space of boundary maps of $(L,M,\Theta)$-quasi-isometric embeddings taking a point $x \in Z$ to a point $o \in Y$ is compact in the uniform topology.
\end{corollary}

\begin{proof}
    By Proposition \ref{prop: Lipchitz}, the space of boundary maps of $(L,M,\Theta)$-quasi-isometric embeddings is equal to the space $\mathcal{S}$ of boundary maps of Lipschitz $(L,M,\Theta)$-quasi-isometric embeddings. 
    The subspace taking $x$ to $y$ is compact in the topology of local uniform convergence on $Z$ by Arzela-Ascoli, and by Proposition \ref{prop: uniform convergence} it is also compact in the topology of uniform convergence on $Z \cup \partial_\infty Z$.
\end{proof}

\subsection{Transversality}

In this section we define some notions that we will need in order to state our criterion for stability, Theorem \ref{thm: finitely nontransverse urqies are stable on balls} below. Specifically, we give definitions of almost-everywhere transversality for measures on a flag variety $\mathcal{F}_\theta$ and for maps to $\mathcal{F}_\theta$. We first recall the definition of Kapovich-Leeb-Millson \cite{KapovichLeebMillson2009ConvexFunctionsSymmetric} of stability of measures on $\partial_\infty Y$.

\medskip

For a convex Lipschitz function $b \colon Y \to \R$, the asymptotic slope of $b$ along a geodesic ray $c \colon [0,\infty) \to Y$ is defined by 
    \[ \slope(b,c) \coloneqq \lim_{t \to \infty} \frac{b \circ c(t)}{t} .\]
A pair of asymptotic geodesic rays have the same asymptotic slopes, so the definition descends to a pairing of convex Lipschitz functions with the visual boundary of $Y$. 
The slope is closely related to the \emph{angular metric} on $\partial_\infty Y$, defined by 
    \[ \angle(\xi,\eta) \coloneqq \sup_{y \in Y} \angle_y(\xi,\eta) \]
where $\angle_y(\xi, \eta)$ is the Riemannian angle at $y$ between the geodesic rays emanating from $y$ asymptotic to $\xi,\eta$. 
In a symmetric space, the angular metric is achieved at $y$ when $y, \xi,\eta$ all belong to a common (maximal) flat, and there always exists a maximal flat containing both $\xi$ and $\eta$ in its visual boundary. 
As a consequence, 
\begin{equation}\label{eqn: slope vs angle}
    - \cos \angle(\xi,\eta) = \slope(b_\xi,\eta)
\end{equation}
holds for all $\xi,\eta \in \partial_\infty Y$. 
When the rank of $Y$ is at least $2$, the angular metric coincides with its associated length metric, which is called the \emph{Tits metric}.

Following Kapovich-Leeb-Millson \cite{KapovichLeebMillson2009ConvexFunctionsSymmetric}, we recall 
\begin{definition} \label{def: stable measure}
    A measure $\nu$ on $\partial_\infty Y$ is called \emph{stable} if for any $\eta \in \partial_\infty Y$, 
        \[ \int_{\partial_\infty Y} \slope(b_\eta,\xi) \nu(\xi) >0 .\]
\end{definition}

Recall that $\Theta$ is closed under the opposition involution and that a pair $\tau,\tau' \in \mathcal{F}_\Theta$ is called \emph{transverse} if there exists a geodesic $c \colon \R \to Y$ so that $c(+\infty)$ is in the relative interior of $\tau$ and $c(-\infty)$ is in the relative interior of $\tau'$.
The \emph{star} of a simplex $\tau$ is the union of all closed ideal Weyl chambers $\sigma$ containing $\tau$.

\begin{definition}
    A point $\eta \in \partial_\infty Y$ is called \emph{transverse} to $\tau \in \mathcal{F}_\Theta$ if it is contained in the star of some $\tau' \in \mathcal{F}_\Theta$ where $\tau'$ is transverse to $\tau$.
\end{definition}

% In other other words, their relative position is the largest element in the Bruhat order of the double coset $ W_\eta \backslash W/W_\Theta$
We will also use the following invariant of the root system $(\mathfrak{a},\Pi)$ and the subset $\Theta$.

\begin{definition} \label{def: separation}
    The \emph{$\Theta$-separation} is
        \[ \epsilon(\mathfrak{a},\Theta) \coloneqq -\cos \min \{ \angle(x,y) : x \in W_\Theta \mathfrak{a}^+, y \in - \mathfrak{a}^+ \} .\]
\end{definition}

\begin{proposition} \label{prop: separation}
    If a $\Theta$-regular point $\xi \in \partial_\infty Y$ is incident to $\tau$ of type $\tau_\Theta$ and $\eta \in \partial_\infty Y$ is transverse to $\tau$, then
        \[ \slope(b_\eta,\xi) > \epsilon(\mathfrak{a},\Theta) .\]
\end{proposition}

\begin{proof}
    The $\Theta$-separation $\epsilon(\mathfrak{a},\Theta)$ is minus cosine of the minimum angle between the star of a simplex $\tau_+$ and the star of any opposite simplex $\tau_-$.
    Since $\xi$ is $\Theta$-regular, its type avoids the walls missing the relative interior of $\tau_\Theta$, so the bound is strict.
\end{proof}

\begin{proposition}\label{prop: separation is nonnegative}
    When $Y$ is irreducible and $\Theta = \Pi$, the diameter of the model Weyl chamber $\sigma_{mod}$ is strictly less than $\frac{\pi}{2}$, so we have the uniformly positive lower bound
    \[ \epsilon(Y) = \epsilon(\mathfrak{a},\Pi) = - \cos(\pi-\diam(\sigma_{mod})) > 0 .\]
    For example $\epsilon(Y_d)=\frac1{d-1}$.

    If $Y$ is not irreducible then $\epsilon(\mathfrak{a},\Pi)=0$ and $\epsilon(\mathfrak{a},\Theta)$ is negative for $\Theta \subsetneq \Pi$. 
\end{proposition}

\begin{proof}
    The angular diameter of $\sigma_{mod}$ is achieved as the angle between two vertices. 
    The vertices of $\sigma_{mod}$ are unit vectors proportional to the fundamental coweights $\omega_i^\vee$, defined by 
        \[ \frac{\langle \omega_i, \alpha_j \rangle}{\langle \alpha_j, \alpha_j \rangle} = \delta_{ij} .\]
    Since $Y$ is irreducible if and only if the root system $(\mathfrak{a},\Pi)$ is irreducible, the stated property can be checked by a straightforward case-by-case calculation.    

    When $Y$ is reducible, the root system splits orthogonally so the diameter of $\sigma_{mod}$ is $\pi/2$. 
\end{proof}

\begin{definition} \label{def: generically transverse}
    A measure $\nu$ on $\mathcal{F}_\Theta$ is called \emph{generically transverse} if, for any $\eta \in \partial_\infty Y$, the set of $\tau \in \mathcal{F}_\Theta$ non-transverse to $\eta$ has $\nu$-measure zero.
\end{definition}

In the case of interest, we aim to deduce stability of a $\Theta$-quasi-isometric embedding purely from properties of its boundary map. 
To this end, the following definition will be useful.

\begin{definition}\label{def: finitely non-transverse}
    A map $\xi \colon \Lambda \to \mathcal{F}_\Theta$ is called \emph{finitely non-transverse} if, for each $\eta \in \partial_\infty Y$, the subset of $\Lambda$ consisting of points $\lambda$ so that $\xi(\lambda)$ is not transverse to $\eta$ is finite. 
\end{definition}

If $\nu$ is a non-atomic finite measure on $\Lambda$ and $\xi$ is finitely non-transverse, then $\xi_*\nu$ is generically transverse.

In particular, finitely-non transverse maps are finite-to-one. 
For us, $\Lambda$ will arise as the visual boundary of a pinched Hadamard space, and $\xi$ will be the boundary map of an $(L,M,\Theta)$-quasi-isometric embedding.

\begin{remark}
    There is a totally geodesic copy of $\H^2 \times \R$ in $Y_3$ and the inclusion of $\H^2 \times \{pt\} \subset \H^2 \times \R \subset Y_3$ is a $\Pi$-quasi-isometric embedding. 
    The boundary map $\xi \colon \partial_\infty \H^2 \to \Flag(\R^3)$ is everywhere non-transverse to either endpoint of $\R$ in $\partial_\infty Y_3$. This shows that finite non-transversality is not automatic for $\Pi$-quasi-isometric embeddings. 
\end{remark}

In the next section, we will extract measures on $\partial_\infty Y$ from $\Theta$-quasi-isometric embeddings. 
We have the following criterion for those measures to be stable, even quantifiably.
Recall the continuous map $\pi_\Theta \colon \partial^{\Theta-reg}_\infty Y \to \mathcal{F}_\Theta$ from the $\Theta$-regular part of the ideal boundary to the flag manifold. 

\begin{lemma}\label{lem: stability of generically transverse measures}
    Let $\nu$ be a locally finite measure on $\partial_\infty Y$ with support in $\partial^{\Theta-reg}_\infty Y$.
    Suppose that the pushforward measure $\pi_\ast \nu$ on $\mathcal{F}_\Theta$ is generically transverse. 
    Then for any $\eta \in \partial_\infty Y$, we have 
        \[ \int_{\partial_\infty Y} \slope(\eta,\xi) \nu(\xi) >  \epsilon(\mathfrak{a},\Theta)\Abs{\nu} .\]
\end{lemma}

\begin{proof}
    Since $(\pi_\Theta)_\ast \nu$ is generically transverse, the subset of $\mathcal{F}_\Theta$ non-transverse to $\eta$ has $(\pi_\Theta)_\ast \nu$-measure zero. 
    Hence $\nu$-almost everywhere on $\partial_\infty Y$, $\mathrm{slope}(\eta,\xi) > \epsilon(\mathfrak{a},\Theta)$ (Proposition \ref{prop: separation}), and so the integral is strictly bounded below by $\epsilon(\mathfrak{a},\Theta)\Abs{\nu}$.
\end{proof}

\subsection{A criterion for stability}

\begin{theorem}\label{thm: finitely nontransverse urqies are stable on balls}
    Let $Y$ be a symmetric space of non-compact type with a non-empty subset $\Theta$ of simple restricted roots, closed under the opposition involution, and let $X$ be a symmetric space of rank $1$.
    Let $\mathcal{QI}_\Theta(L,M)$ be the space of $L$-Lipschitz $(M,\Theta)$-quasi-isometric embeddings from $X$ to $Y$ and let $\mathcal{Q}$ be a closed subspace of $\mathcal{QI}_\Theta(L,M)$, invariant by pre- and post-composition by isometries, such that all $f \in \mathcal{Q}$ have finitely non-transverse boundary map.
    Then
    \begin{equation} \label{eqn: purqies stable on balls}
        \lim_{r_0 \to \infty} \inf_{f,x,\eta, r \geq r_0} \frac{1}{r} \int_{\partial B(x,r)} b_\eta \circ f(z) - b_\eta \circ f (x) \mu_x^r(z) > \frac{\epsilon(\mathfrak{a},\Theta)}{M},
    \end{equation}
    where the infimum is over all $x \in X$, $\eta \in \partial_\infty Y$, and $f \in \mathcal{Q}$, and $\mu_x^r$ is the hitting measure on $\partial B(x,r)$, and $\epsilon(\mathfrak{a},\Theta)$ is the $\Theta$-separation (Definition \ref{def: separation}).
\end{theorem}

\begin{proof}
We take a sequence of $f_n$, $x_n$, $r_n \to \infty$, and $\eta_n$ converging to the liminf. 
Up to pre- and post-composition by isometries, we can suppose the $x_n$ are all equal to some point $x \in X$, and moreover $f_n(x)$ is always equal to some point $y \in Y$.

For each $n$ we have
\begin{align}
    &\frac{1}{r_n} \int_{X} (b_{\eta_n}\circ f_n(z) - b_{\eta_n}\circ f_n(x))\mu_x^{r_n}(z) \\
    =& \int_{X} \frac{b_{\eta_n}\circ f_n(z) - b_{\eta_n}\circ f_n(x)}{d_Y(f_n(x),f_n(z))} \frac{d_Y(f_n(x),f_n(z))}{t_n}\mu^{r_n}_x(z) \\
    =& \int_{Y} \frac{b_{\eta_n}(w) - b_{\eta_n}(y)}{d_Y(y,w)} ((f_n)_* \nu_n)(w) \label{eqn: finite n stability}
\end{align}
where $\nu_n$ is the measure on $X$ defined by
    \[
        \nu_n(z) = \frac{d_Y(f_n(x),f_n(z))}{r_n} \mu_x^{r_n}(z).
    \]
Since $f_n$ is an $L$-Lipschitz $(M,\Pi)$-quasi-isometric embedding, we have
\begin{equation}\label{eqn: bounding meaure from urqie}
   \left(\frac1{M}-\frac1{r_n}\right) \mu_x^{r_n} \leq \nu_n \leq L \mu_x^{r_n}. 
\end{equation}

In particular, $\frac{1}{M} - \frac{1}{r_n} \le \Abs{\nu_n} \le L$, so the sequence $\nu_n$ is weak-* precompact on $X \cup \partial_\infty X$, and letting $\nu$ be a subsequential weak-* limit, $\nu$ has mass at least $\frac{1}{M}$. 
Furthermore, $\mu_x^{r_n}$ limits to the harmonic measure $\mu_x^\infty$ supported on $\partial_\infty X$ so by \eqref{eqn: bounding meaure from urqie}, $\nu$ is also supported on $\partial_\infty X$ and is in the measure class of $\mu_x^\infty$.

Passing to a further subsequence if necessary, assume also that 
\begin{enumerate}
    \item the sequence $(f_n)_\ast \nu_n$ converges weakly to a measure $\nu'$ on $Y \cup \partial_\infty Y$,
    \item $\eta_n$ converges to a point $\eta$ of $\partial_\infty Y$, and
    \item $f_n$ converges uniformly on compacts to an $L$-Lipschitz $(M,\Theta)$-quasi-isometric embedding $f$ in $\mathcal{Q}$.
\end{enumerate}

We now verify that the conditions of Lemma \ref{lem: stability of generically transverse measures} are satisfied. 
Since $f_n$ are all $(L,M,\Theta)$-quasi-isometric embeddings, $\nu'$ has support inside $\partial_\infty^{\Theta-reg} Y$. 
Since the maps
\[
    \overline{f}_n = f_n \cup \partial f_n \colon X \cup \partial_\infty X \to Y \cup \mathcal{F}_\Theta
\]
converge uniformly to $\overline{f}$ by Proposition \ref{prop: uniform convergence}, 
and the map
\[
    \mathrm{Id} \cup \pi \colon Y \cup \partial^\mathrm{\Theta-reg}_\infty Y \to  Y \cup \mathcal{F}_\Theta
\]
is continuous, we have
\[
    \pi_\ast \nu' = (\mathrm{Id} \cup \pi)_\ast \lim ((f_n)_\ast \nu_n) = \lim ((f_n \cup \partial f_n)_\ast \nu_n) = (f \cup \partial f)_\ast \nu = (\partial f)_\ast \nu.
\]
Since $\nu$ is in the measure class of the harmonic (round) measure $\mu_x^\infty$, it has no atoms.
Then its pushforward by the finitely non-transverse map $\partial f$ is generically transverse.
So Lemma \ref{lem: stability of generically transverse measures} applies and yields
\begin{equation} \label{eqn: slope lower bound}
    \int_{\partial_\infty Y} \slope(b_\eta,\xi) \nu'(\xi) > \Abs{\nu'} \epsilon(\mathfrak{a},\Theta) \ge \frac{\epsilon(\mathfrak{a},\Theta)}{M}.
\end{equation}

Finally, we relate this to equation \eqref{eqn: finite n stability}. It is natural to consider the function $\partial_\infty Y \times ((Y\setminus\{y\}) \cup \partial_\infty Y)  \to \R$ defined for $b \in \partial_\infty Y$, $w \in Y$ and $\xi \in \partial_\infty Y$ by 
    \[ (b,w) \mapsto \frac{b(w)-b(y)}{d(y,w)}, \quad (b,\xi) \mapsto \slope(b,\xi) .\]
It is a standard fact that this function is lower semi-continuous, see \cite[Lemma 9.16]{bridsonMetricSpacesNonPositive1999}.
Since $(f_n)_\ast \nu_n$ weakly limits to $\nu'$, this lower semi-continuity implies, by Portmanteau's theorem, that:
\begin{equation} \label{eqn: Portmanteau}
     \liminf_n \int_Y \frac{b_{\eta_n}(w) - b_{\eta_n}(y)}{d_Y(y,w)} ((f_n)_* \nu_n)(w) \ge \int_{\partial_\infty Y} \slope(\eta,\xi)\nu'(\xi) .
\end{equation}
Together with \eqref{eqn: slope lower bound}, this implies the theorem.

\end{proof}

\medskip

Let $f: X \to Y$ be a continuous map between symmetric spaces, and suppose $f(x) = y$ for some $x\in X$ and $y \in Y$. For the following corollary, we call a continuous map $g: X \to Y$ a \emph{limit} of $f$ if there is a sequence of isometries $\gamma_i \in \mathrm{Isom}(X)$ and $\gamma'_i \in \mathrm{Isom}(Y)$ such that $\gamma'_i \circ f \circ \gamma_i (x) = y$ and $\gamma'_i \circ f \circ \gamma_i$ converges locally uniformly to $g$. 

\begin{corollary}\label{cor: positive separation and finitely non-transverse limits}
    There exists $r=r(L,M,X,Y)$ such that if $f \colon X \to Y$ is an $L$-Lipschitz $(M,\Pi)$-quasi-isometric embedding such that every limit of $f$ has finitely non-transverse boundary map, then $f$ is stable at scale $r$ (Definition \ref{def: stable at scale r}).
\end{corollary}

\begin{proof}
    Suppose that $f$ were not stable. 
    In view of Proposition \ref{prop: Busemann enough}, there are sequences $x_i$, $\eta_i$, and $r_i \to \infty$ such that
    \[
    \int_X b_{\eta_i} \circ f(z) - b_{\eta_i} \circ f(x_i) \mu^{r_i}_{x_i}(z) < 1
    \]
     where $\mu^{r_i}_{x_i}$ is the hitting measure.
     Then the liminf \eqref{eqn: purqies stable on balls} would be less than or equal to zero, which contradicts Theorem \ref{thm: finitely nontransverse urqies are stable on balls}, since $\epsilon(\mathfrak{a},\Pi)$ is nonnegative by Proposition \ref{prop: separation is nonnegative}.
\end{proof}

\begin{remark}
   The same proof applies whenever $\epsilon(\mathfrak{a},\Theta)$ is nonnegative. 
\end{remark}

In our main existence theorem \ref{thm: Donalson-Corlette}, the bound on the distance between an $L$-Lipschitz map $f$ that is stable at scale $r$ and the harmonic map depends only on $L$, $r$, $X$, and $Y$. Hence from Corollary \ref{cor: positive separation and finitely non-transverse limits} we deduce: 

\begin{corollary} \label{cor: James}
    There exists $R = R(L,M,X,Y)$ such that if $f: X \to Y$ is an $L$-Lipschitz $(M,\Pi)$-quasi-isometric embedding such that every limit of $f$ has finitely non-transverse boundary map, then there exists a harmonic map $h: X \to Y$ at distance at most $R$ from $f$.
\end{corollary}

If $Y$ is Gromov hyperbolic, then the issue of transversality between points in the visual boundary of $Y$ simplifies considerably; any two distinct points are transverse. 
Using Theorem \ref{thm: finitely nontransverse urqies are stable on balls}, we recover in this case:

\begin{corollary}[{Benoist-Hulin \cite[Theorem 1.1]{BH1}, Li-Wang \cite[Theorem 2.3]{liHarmonicRoughIsometries1998}}]\label{cor: benoist-hulin}
    Let $X$ and $Y$ be rank $1$ symmetric spaces and let $f \colon X \to Y$ be a quasi-isometric embedding. 
    Then there exists a unique harmonic map $h \colon X \to Y$ at bounded distance from $f$. Moreover the bound depends only on $X,$ $Y,$ and the constants of quasi-isometry.
\end{corollary}

\begin{proof}
    Let $f'$ be the mollification of $f$ from Proposition \ref{prop: Lipchitz}.
    Note that the quasi-isometry constants of $f$ control the distance from $f$ to $f'$ as well the quasi-isometry constants of $f'$.
    
    The classical Morse lemma implies that $f '\colon X \to Y$ extends to a continuous, injective boundary map $\partial f' \colon \partial_\infty X \to \partial_\infty Y = \mathcal{F}_\Pi$. 
    Since any two distinct points are transverse, this implies the finite non-transversality property for $f'$ and all its limits. 
    Also, in rank 1, $\epsilon(\mathfrak{a}, \Pi) = 1$. Hence existence follows from Theorem \ref{thm: finitely nontransverse urqies are stable on balls}. Uniqueness follows from Theorem \ref{thm: uniqueness}. The dependence of the bound follows from Corollary \ref{cor: James}.
\end{proof}

\section{The universal Hitchin component of \texorpdfstring{$\PGL_d(\R)$}{PGLdR}}\label{sec: universal hitchin}

In this section we study boundary maps of $\Pi$-quasi-isometric embeddings from $\H^2$ to the symmetric space $Y_d$ of $\PGL_d(\R)$.

In the first part of this section, we discuss positivity of configurations of flags, and the finite non-transversality of positive circles in the flag variety. A recent theorem of Saldanha-Shapiro-Shapiro in fact gives an explicit upper bound on the number.

In the remainder of this section, we discuss cross ratios and quasisymmetry of positive configurations of flags, and prove Theorem \ref{thm: main quasisymmetric} that the universal Hitchin component can be described purely in terms of boundary maps.

\subsection{Positivity in \texorpdfstring{$\Flag(\R^d)$}{Flag(Rd)}}

Let $N$ be the subgroup of $\SL_d(\R)$ of strictly upper triangular matrices:
    \[ N \coloneqq \left\{ 
    \begin{bmatrix}
                1 & \ast   & \cdots   & \ast \\
                 & \ddots &  &  \vdots \\
                 &      & 1      & \ast \\
                0 &       &      & 1
    \end{bmatrix} \right\} .\]
An element $n \in N$ is called \emph{totally positive} if all if its minors which can be positive are positive. 
In particular, all the ``right-justified" minors, i.e.\ those involving any $k$ rows and the last $k$ columns, are positive. 

We let 
    \[ \sigma_\infty = \langle e_1 \rangle \subset \langle e_1,e_{2} \rangle \subset \cdots \subset \langle e_1,e_{2}, \dots, e_{d-1} \rangle \]
denote the standard ascending flag, and let
    \[ \sigma_0 = \langle e_d \rangle \subset \langle e_d,e_{d-1} \rangle \subset \cdots \subset \langle e_d,e_{d-1}, \dots, e_{2} \rangle \]
denote the standard descending flag. 
Note that $N$ is the unipotent radical of the stabilizer of $\sigma_\infty$ in $\SL_d(\R)$. 
An arbitrary triple of flags $(\sigma_1,\sigma_2,\sigma_3)$ is called \emph{positive} if there exists $g \in \PGL_d(\R)$ and a totally positive $n \in N$ such that $g(\sigma_1,\sigma_2,\sigma_3) = (\sigma_0, n \sigma_0, \sigma_\infty)$. 
\begin{definition}[{\cite{fock2006moduli}}]\label{def: positive}
    A continuous map $\xi \colon \RP^1 \to \Flag(\R^d)$ is called \emph{positive} if it takes every pairwise distinct triple to a positive triple of flags. 
\end{definition}

The following fact is standard, and the proof works for any flag manifold $\mathcal{F}_\Theta = G/P_\Theta$ supporting a positive structure in the sense of Guichard-Wienhard \cite{gw2025positivity}.
Recall that a continuous map $\xi \colon \RP^1 \to \Flag(\R^d)$ is called transverse when it takes distinct pairs to transverse flags.

\begin{lemma} \label{lem: positive connected}
    Suppose a sequence $\xi_n$ of positive maps $\RP^1 \to \Flag(\R^d)$ limits pointwise to a transverse map $\xi \colon \RP^1 \to \Flag(\R^d)$. 
    Then $\xi$ is positive.
\end{lemma}

\begin{proof}
    Let $x_1,x_2,x_3$ be a pairwise distinct triple in $\RP^1$. 
    Since transversality is an open condition, there exist path-connected neighborhoods $B_i$ of $\xi(x_i)$, for $i=1,2,3$, which are small enough that each element of $B_i$ is transverse to each element of $B_j$ for $i \ne j$. 
    For $n$ large enough, $\xi_n(x_i)$ lies in $B_i$ and we may choose paths $c_i$ in $B_i$ from $\xi_n(x_i)$ to $\xi(x_i)$.
    By \cite[Lemma 3.6]{guichard2021positivityrepresentationssurfacegroups}, the triple $\xi(x_1),\xi(x_2),\xi(x_3)$ is positive.
    If $x_2'$ is in the same component of $\RP^1 \setminus \{x_1,x_3\}$ as $x_2$, then another application of \cite[Lemma 3.6]{guichard2021positivityrepresentationssurfacegroups} allows us to conclude that the triple $\xi(x_1),\xi(x_2'),\xi(x_3)$ is positive. 
    Iterating this procedure at $x_1$ and $x_3$ shows that any pairwise distinct triple of $\RP^1$ in the same cyclic order as $(x_1,x_2,x_3)$ maps to a positive triple by $\xi$.
    Since the positivity of a triple is independent of its ordering, see \cite[Proposition 2.11]{guichard2021positivityrepresentationssurfacegroups}, this implies that $\xi$ is positive.
\end{proof}

The fact that positive maps yield generically transverse measures is a direct consequence of a theorem of Saldanha-Shapiro-Shapiro.
They prove that, if $\Lambda$ is the image of a positive map, then each vertex $v$ of $\partial_\infty Y$ is non-transverse to $\Lambda$ at finitely many points. 
In fact, they give an explicit upper bound.

\begin{theorem}
[{\cite{SaldanhaShapiroShapiro2021Grassmannconvexityrevisited,SaldanhaShapiroShapiro2023finitenessofrank}}]
\label{thm: Saldanha-Shapiro-Shapiro}
    For any $1 \le k \le d$, there exists an explicit $C(k,d)$ so that:

    If $\Lambda \subset \Flag(\R^d)$ is the image of a positive map and $v \in \Gr(k,\R^d)$, then $v$ is non-transverse to at most $C(k,d)$ points of $\Lambda$.
\end{theorem}

\begin{remark}
    The Grassmann convexity conjecture \cite{ShapiroShapiro2000grassmannconvexityRP3} predicts that $C(k,d)$ can be taken to be $k(d-k)$.
\end{remark}

\subsection{Proof of Theorem \ref{thm: main schoen}}

Let $Y_d$ be the symmetric space of $\PGL_d(\R)$.

\begin{corollary}[Theorem \ref{thm: main schoen}] \label{cor: schoen}
    Every $\Pi$-quasi-isometric embedding $\H^2 \to Y_d$ with positive boundary map is within bounded distance of a unique harmonic map.
\end{corollary}

\begin{proof} 
    Let $f \colon \H^2 \to Y_d$ be an $L$-coarsely Lipschitz $(M,\Pi)$-quasi-isometric embedding with positive boundary map $\partial f \colon \partial_\infty \H^2 \to \Flag(\R^d)$.     By Proposition \ref{prop: Lipchitz}, $f$ is at bounded distance from an $L$-Lipschitz map $f'$ with the same boundary map, so it suffices to prove that $f'$ is stable.

    By Lemma \ref{lem: positive connected} and Corollary \ref{cor: compactness of urqies}, the boundary maps of all limits of $f$ are still positive. So by Theorem \ref{thm: Saldanha-Shapiro-Shapiro} of Saldanha-Shapiro-Shapiro, $\partial f'$, and all its limits, have the finite non-transversality property, and we may apply Corollary \ref{cor: positive separation and finitely non-transverse limits}.

    Then there exists $r=r(L,M,d)$ so that $f'$ is stable at scale $r$.
    By Theorem \ref{thm: Donalson-Corlette}, there exists a harmonic map $h \colon \H^2 \to Y_d$ at distance $R=R(L,r)$ from $f'$. 
    By Theorem \ref{thm: uniqueness}, the harmonic map $h$ is unique. 
\end{proof}

Recall (Definition \ref{def: UHitd}) that we define the universal Hitchin component $\mathrm{UHit}_d$ to be the space of $\Pi$-quasi-isometric embeddings from $\H^2$ to $Y_d$ with positive boundary map from $\RP^1$ to $\mathrm{Flag}(\R^d)$, up to bounded distance. With this definition, Corollary \ref{thm: main schoen} says:

\begin{corollary}
    The universal Hitchin component is equal to the space of harmonic $\Pi$-quasi-isometric embeddings $\H^2 \to Y_d$ with positive boundary map.
\end{corollary}

Using heat kernel estimates on $\H^2$, we also obtain a positive lower bound on the drift (Definition \ref{def: drift of a map}) of an $(M,\Pi)$-quasi-isometric embedding.

\begin{corollary}
    Every $(L,M,\Pi)$-quasi-isometric embedding from $\H^2$ to $Y_d$ with positive boundary map has drift at least $\frac{1}{M(d-1)}$.
\end{corollary}

\begin{proof}
    This follows from Theorem \ref{thm: finitely nontransverse urqies are stable on balls} together with the following estimate for the heat kernel on $\H^2$ \cite[Theorem 3.1]{DaviesMandouvalos1988heatkernelhyperbolic}:
    \[ p_t(\rho) \asymp t^{-1/2} \left(1+\rho\right) \left(1+t+\rho\right)^{-1/2} e^{- \frac1{4} t - \frac12 \rho - \frac{\rho^2}{4t}}
    \]
    uniformly for $0 \le \rho < \infty$ and $0 < t < \infty$.
\end{proof}

Extending our techniques to other higher Teichm\"uller spaces hinges on the following question.

\begin{question}
    Suppose $(G,\Theta)$ has a positive structure in the sense of Guichard-Wienhard \cite{gw2025positivity}.
    Is every continuous positive map $\xi \colon \RP^1 \to \mathcal{F}_\Theta$ finitely non-transverse?
\end{question}

For example, when $G=\SO_0(2,n+1)$, for $n \ge 0$, there is a subset $\Theta$ of $\Pi$ so that $(G,\Theta)$ admits a positive structure and moreover $\mathcal{F}_\Theta = \Ein^{1,n}$, the Einstein universe. When $n \ge 2$, the restricted root system has type $B_2$ and $\Theta$ is the singleton consisting of the long simple root. Using the fact that positive maps $\xi \colon S^1 \to \Ein^{1,n}$ are the same as spacelike maps in the conformal Lorentzian structure, it is not hard to verify that they satisfy the finite non-transversality property. Then, by following the proof of Corollary \ref{cor: schoen} with the previous sentences replacing Theorem \ref{thm: Saldanha-Shapiro-Shapiro}, any $\Theta$-quasi-isometric embedding $f \colon \H^2 \to \SO_0(2,n+1)/\SO(2)\times \SO(n+1)$ with positive boundary map $\partial f \colon \partial_\infty \H^2 \to \Ein^{1,n}$ is at bounded distance from a unique harmonic map.

\subsection{Positive quasi-symmetric maps}

We always consider $\partial_\infty \H^2 = \RP^1$. 

\subsubsection{\texorpdfstring{$\PGL_2(\R)$}{PGL2R}}

The cross ratio of four points in $\RP^1$ is 
\begin{equation} \label{eqn: cross ratio}
    CR(\lambda_1,\lambda_2,\lambda_3,\lambda_4) = \frac{(\lambda_4 - \lambda_3)(\lambda_2 - \lambda_1)}{(\lambda_3 - \lambda_2)(\lambda_1 - \lambda_4)}
\end{equation}
With this convention 
\[
    CR(x,0,1,\infty) = x.
\]
In particular the cross ratio of four dihedrally ordered points is negative. 
Let $(\RP^1)^{(k)}$ denote the space of pairwise distinct $k$-tuples in $\RP^1$, and $(\RP^1)^{(k)+}$ the space of dihedrally ordered pairwise distinct $k$-tuples. 
If $\phi$ is a homeomorphism from $\RP^1$ to $\RP^1$, we write $\phi^k$ for the induced self-homeomorphism of $(\RP^1)^{(k)}$.

A homeomorphism $\phi: \RP^1 \to \RP^1$ is called \emph{$K$-quasisymmetric} if 
\begin{equation} \label{eqn: K quasi-symmetric to RP1}
    \frac1{K} \leq -CR(\phi^4(\tau)) \leq K
\end{equation}
for all $\tau \in (\RP^1)^{(4)}$ such that $-CR(\tau) = 1$.

On the other hand a homeomorphism $\phi: \R \to \R$ is called $K$-quasisymmetric if
\begin{equation} \label{eqn: K quasi-symmetric to R}
    \frac{1}{K} \le \frac{\phi(x+t) - \phi(x)}{\phi(x)-\phi(x-t)} \le K
\end{equation}
Since $CR(x-t,x,x+t,\infty)=-1$, the restriction of a $K$-quasisymmetric self-homeomorphism of $\RP^1$ fixing $\infty$ to a self-homeomorphism of $\R$ is still $K$-quasisymmetric. 
On the other hand, it is a standard lemma in the theory of quasi-symmetric maps that if $\phi\colon \R \to \R$ is $K$-quasisymmetric then the extension to a homeomorphism of $\RP^1$ is $K'$-quasisymmetric for some explicit $K'(K)$.

We also define, for $\tau \in (\RP^1)^{(4)+}$ the log-cross-ratio
\[
    cr(\tau) \coloneqq \log(-CR(\tau)).
\]

\begin{proposition}[e.g.{\cite[Theorem 11.3]{Heinonen2001analysisonmetricspaces}}]\label{prop: cr is coarse linear}
    For every $K$ there exist $L,M$ so that:

    If $\phi \colon \RP^1 \to \RP^1$ is $K$-quasisymmetric, then for every $\tau \in (\RP^1)^{(4)+}$, 
        \[ \frac1{M} cr(\tau) - 1 \le cr(\phi^4(\tau)) \le L cr(\tau) + L .\]
\end{proposition}

From this lemma follows both the H\"{o}lder continuity of quasisymmetric maps, and the compactness of the family of normalized $K$-quasisymmetric maps, i.e.\ those satisfying $\phi^3(0,1,\infty) = (0,1,\infty)$.

\subsubsection{\texorpdfstring{$\PGL_d(\R)$}{PGLdR}}

In order to define quasi-symmetric maps to the flag variety $\mathrm{Flag}(\R^d)$, we first generalize the cross-ratio and log-cross-ratio to four-tuples of flags.

A dihedrally ordered $k$-tuple of flags is called positive if it can be extended to a positive map of the circle. Any pairwise-transverse four-tuple of flags is equivalent, up to the action of $\mathrm{PGL}_d(\R)$, to one of the form
\begin{equation} \label{eqn: standard 4-tuple}
    (m \cdot \sigma_0, \sigma_0, n \cdot \sigma_0, \sigma_\infty).
\end{equation}
where $m$ and $n$ are unipotent upper triangular matrices. The four-tuple is positive if $m$ and $n$ can be chosen such that $n$ and $m^{-1}$ are totally positive.

\begin{definition}
    For $i = 1, \ldots, d-1$, we define the \emph{$i$th cross-ratio} of a positive 4-tuple of the form \eqref{eqn: standard 4-tuple} by
    \[
    CR_i(m \cdot \sigma_0, \sigma_0, n \cdot \sigma_0, \sigma_\infty) = m_{i,i+1}/n_{i,i+1}.
    \]
    We extend this to all positive four-tuples by $\mathrm{PGL}_d(\R)$-invariance.
\end{definition}

Noting that $CR_i$ is always negative, we define:

\begin{definition}
    For $i= 1, \ldots, d-1$, the \emph{$i$th log-cross ratio} of a positive 4-tuple $\tau$ is
    \[
    cr_i(\tau) = \log (- CR_i(\tau)).
    \]
\end{definition}

\begin{remark}
    Our cross ratio is not a cross ratio in the sense of Labourie \cite{labourie2007crossratios}.
\end{remark}

We give another interpretation of this cross ratio in terms of projections from triples of flags to the symmetric space. Let $H$ be the Cartan subgroup of diagonal matrices. Then $H$ acts on the space of positive triples $(\sigma_0,n\cdot \sigma_0, \sigma_\infty)$ by conjugation on the matrix $n$. Since $H$ acts simply transitively on the superdiagonal of $n$, there is always a unique representative of the triple such that $n_{i,i+1} = 1$ for all $i$.
\begin{definition}
    We say that the positive triple $(\sigma_0,n\cdot \sigma_0, \sigma_\infty)$ is \emph{normalized} if $n_{i,i+1} = 1$ for all $i$.
\end{definition}

Since the stabilizer in $\mathrm{PGL}_d(\R)$ of a normalized triple of flags is trivial, any positive triple of flags determines a unique element of the group $G$, and so we may choose any $o$ in the flat $F(\sigma_0,\sigma_\infty)$ and define:

\begin{definition} \label{def: pd}
    The \emph{projection $p_d \colon \mathrm{Flag}(\R^d)^{(3)+} \to Y$} associated to the set of cross ratios $CR_i$ is the map sending
    \[
    (g \cdot \sigma_0, g n \cdot \sigma_0, g \cdot \sigma_\infty)
    \]
    to $g \cdot o$ for any normalized triple $(\sigma_0,n\cdot \sigma_0, \sigma_\infty)$.
\end{definition}

Note that $p_d(\sigma,\sigma', \sigma'')$ always lies in the flat determined by $\sigma$ and $\sigma''$. 
Since the connected component $H_0 = \exp(\mathfrak{a})$ of $H$ acts simply transitively on the flat determined by $\sigma_0$ and $\sigma_\infty$, for any pair of positive triples 
\[
(\sigma_0, n \cdot \sigma_0, \sigma_\infty), (\sigma_0, m \cdot \sigma_0, \sigma_\infty)
\]
there is always a unique element $v \in \mathfrak{a}$ such that 
\[
\mathrm{exp}(v) \cdot p(\sigma_0, n \cdot \sigma_0, \sigma_\infty) = p(\sigma_0, m \cdot \sigma_0, \sigma_\infty).
\]
Moreover, when the quadruple 
    \[(m \cdot \sigma_0, \sigma_0, n \cdot \sigma_0, \sigma_\infty) \]
is positive, the relationship with the cross-ratio is given by:
\begin{prop} \label{prop: cross ratios and projections}
    The log cross-ratio $cr_i$ is given by applying the $i$th root $\alpha_i$ to the vector $v$.
\end{prop}

With our generalized cross ratios in hand, the obvious definition of quasisymmetric is:

\begin{definition}
    A positive map $\phi\colon \RP^1 \to \mathrm{Flag}(\R^d)$ is \emph{$K$-quasisymmetric} if for every four-tuple $\tau \in (\RP^1)^{(4)}$ with $CR(\tau) = -1$, 
    \[
    1/K \leq -CR_i(\phi^{4}\tau) \leq K.
    \]
\end{definition}

Let $\phi\colon \RP^1 \to \mathrm{Flag}(\R^d)$ be a normalized positive map. 
Then for all $t\in\R$,
\[
    \phi(t) = n(t)\phi(0)
\]
where $n(t) \in N$, the unipotent radical of the stabilizer of $\sigma_\infty$. 
Writing
    \[ n(t) = 
        \begin{bmatrix}
                1 & n_{1,2}(t) & * & *  \\
                  & \ddots & \ddots & * \\
                  &        & 1 & n_{d-1,d}(t) \\
                  &        & & 1
        \end{bmatrix}, \]
we denote by $n^1(t)$ the superdiagonal, i.e.\ the vector of entries of $n$ one above the diagonal. 
From positivity, one sees that each component $n_{i,i+1}$ of $n^1(t)$ is a monotonic function of $t$, and from continuity, one deduces that in fact each component of $n^1(t)$ is a homeomorphism from $\R$ to $\R$. The map $\phi$ is normalized if and only if each homeomorphism $n_{i,i+1}$ is normalized in the sense that it sends 0 to 0 and 1 to 1.

From the definition of our generalized cross-ratios, one sees immediately that if $\phi$ is $K$-quasisymmetric, then each $n_{i,i+1}$ is a $K$-quasisymmetric homeomorphism from $\R$ to $\R$.

In fact, one can recover a normalized positive map from the $d-1$ homeomorphisms $n_{i,i+1}$. This is the content of the following proposition.

\begin{prop} \label{prop: d-1 qs homeos}
    The space of normalized continuous positive maps $\phi \colon \RP^1 \to \Flag(\R^d)$ is in bijection with the space of $d-1$-tuples $(\phi_1,\ldots, \phi_{d-1})$ of normalized homeomorphisms from $\RP^1$ to $\RP^1$.
\end{prop}

For differentiable curves, Fock-Goncharov observed this in \cite[Proposition 7.2]{fock2006moduli}.

\begin{proof}[Sketch of proof]
    From $\phi$, we define $(\phi_1,\ldots,\phi_{d-1}) = n^1$. We obtain the converse by solving the Maurer-Cartan equation
    \[
    n^{-1}dn =         \begin{bmatrix}
                0 & d\phi_1 & 0 & 0  \\
                  & \ddots & \ddots & 0 \\
                  &        & 0 & d\phi_{d-1} \\
                  &        & & 0
        \end{bmatrix} 
    \]
    Note that the differentials $d\phi_i$ are only measures, but this is fine because the solution can be written in integrals, using the ordered exponential:
    \[
    n_{i,i+j+1}(t) = \int_{0 \leq t_0 \leq \cdots \leq t_j \leq t} d\phi_i \otimes d\phi_{i+1} \otimes \cdots \otimes d\phi_{i+j}
    \]
    One then checks positivity of the flag minors of $n(t)$. To see this, one uses the Cauchy-Binet formula
    \[
    d[n]_{i_1,\ldots, i_k}^{j_1,\ldots, j_k} = \sum_{l=1}^k [n]_{i_1,\ldots, i_l + 1, \ldots, i_k}^{j_1,\ldots, j_k} d \phi_{i_l + 1}
    \]
    in which the index $i_l$ has been replaced by $i_l + 1$, and the $[n]$ expressions denote the determinant of minors. Inductively, this expresses each minor as an integral of non-negative functions, which implies non-negativity of the flag minors. Strict positivity comes from the positivity of determinants of diagonal minors at time zero, when $n(0)$ is the identity.

    We have shown that for each $s<t$, the triple $(\phi(s),\phi(t),\sigma_\infty)$ is positive.
    In particular, $\phi$ sends distinct pairs to transverse pairs and sends a pairwise distinct triple to a positive triple. 
    Hence, $\phi$ is positive.
\end{proof}

\subsection{Quasiperiodicity}

\begin{definition}
    A family $\mathcal{F}$ of positive maps from $\RP^1$ to $\mathrm{Flag}(\R^d)$ is \emph{quasiperiodic} if it is invariant under the action of $\mathrm{PGL}_2(\R) \times \mathrm{PGL}_d(\R)$ and the subfamily of normalized maps is compact with respect to uniform convergence.
\end{definition}

By Corollary \ref{cor: compactness of urqies}, the family of positive maps arising as boundary maps of $L$-Lipschitz $(M,\Pi)$-quasi-isometric embeddings is quasiperiodic.

From the continuity of the cross-ratio on the space of positive 4-tuples, one sees immediately that any quasi-periodic family is $K$-quasisymmetric for some $K$. In particular,

\begin{prop} \label{prop: piqie implies qs}
    There is a $K=K(L,M,d)$ such that every positive map $\RP^1 \to \mathrm{Flag}(\R^d)$ arising as the boundary map of an $(L,M,\Pi)$-quasi-isometric embedding $\H^2 \to Y$ is $K$-quasisymmetric.
\end{prop}

Our other example of a quasi-periodic family is the family of $K$-quasisymmetric positive maps.

\begin{proposition}\label{prop: compactness of qs maps}
    The family of $K$-quasisymmetric positive maps $\RP^1 \to \Flag(\R^d)$ is quasi-periodic.
\end{proposition}

\begin{proof}
    Let $\phi_k$ be a sequence of normalized $K$-quasisymmetric positive maps. Using compactness of the space of the space of quasi-symmetric maps from $\R$ to $\R$, we can pass to a subsequence on which the homeomorphisms $n_{i,i+1}: \R \to \R$ all converge locally uniformly. By the construction of Proposition \ref{prop: d-1 qs homeos}, this guarantees that the subsequence of maps to the flag variety also converges locally uniformly on $\R$. On the other hand, applying the same reasoning to $\phi_k(-1/\lambda)$, we can pass to a further subsequence which converges in a neighborhood of infinity. Hence on this further subsequence, we have uniform convergence on all of $\RP^1$.
\end{proof}

We will use the following lemma twice. The proof is obvious, just like the proof of Proposition \ref{prop: piqie implies qs}. 

\begin{lemma} \label{lem: compact triples}
    Let $\mathcal{F}$ be a quasi-periodic family of maps. Let $B$ be a compact subset of $(\RP^1)^{(3)}$. Then the space of triples of flags $\phi^3(\tau)$ where $\phi$ ranges over all normalized maps in $\mathcal{F}$ and $\tau \in B$, is compact in the space of positive triples of flags.
\end{lemma}

\subsection{Extending quasisymmetric maps}

Recall that $Y_d$ is the symmetric space for $\PGL_d(\R)$. The following two theorems allow us to establish a correspondence between equivalence classes of positive $\Pi$-quasi-isometric embeddings from $\H^2$ to $Y_d$ and their boundary maps.

\begin{theorem}\label{thm: from qs map to purqie}
    For every $K$ there exist $L=L(K)$ and $M=M(K)$ so that:
    
    If $\xi \colon \RP^1 \to \Flag(\R^d)$ is positive $K$-quasisymmetric then there exists $f \colon \H^2 \to Y_d$ which is an $L$-coarsely Lipschitz $(M,\Pi)$-quasi-isometrically embedding and $\partial f = \xi$.
\end{theorem}

Such maps are unique up to bounded distance:

\begin{theorem}\label{thm: purqies with same boundary}
    For every $L,M$ there exists $D=D(L,M)$ so that:
    If $f_1,f_2 \colon \H^2 \to Y_d$ are $L$-coarsely Lipschitz $(M,\Pi)$-quasi-isometrically embeddings with the same positive boundary map, then the distance from $f_1$ to $f_2$ is bounded by $D$.
\end{theorem} 

\subsubsection{Uniqueness}

We begin with uniqueness, Theorem \ref{thm: purqies with same boundary}. 
We prove two lemmas first. 

Let $\rho$ be the half-sum of all positive roots: 
    \[ \rho = \frac12 \sum_{\alpha \in \Pi} \alpha ,\]
sometimes called the Weyl vector, which lies in the interior of $\mathfrak{a}^+$ and is invariant by the opposition involution (these are the only two properties of $\rho$ we will use).

For $\sigma \in \Flag(\R^d)$ we write $\rho(\sigma)$ for the element of $\partial_\infty Y$ with type $\rho$ in the chamber corresponding to $\sigma$. 
For $\xi \in \partial_\infty Y$ we write $\delta_\xi$ for the Dirac measure of $\xi$ on $\partial_\infty Y$. 
We may sometimes suppress $\rho$, and write $b_\sigma = b_\rho(\sigma)$ for the Busemann function associated to $\rho(\sigma) \in \partial_\infty Y$.

\begin{lemma}\label{lem: positive triples are stable}
    Let $(\sigma_1,\sigma_2,\sigma_3)$ be a positive triple of flags. 
    Then the sum $\delta_{\rho(\sigma_1)} + \delta_{\rho(\sigma_2)} + \delta_{\rho(\sigma_3)}$ is a stable measure on $\partial_\infty Y$ in the sense of Definition \ref{def: stable measure}. 
    In particular, $b_{\rho(\sigma_1)} + b_{\rho(\sigma_2)} + b_{\rho(\sigma_3)}$ is proper and convex.
\end{lemma}

\begin{proof}
    Write $b_i = b_{\rho(\sigma_i)}$ for $i=1,2,3$.
    We need to show that for each $\eta \in \partial Y$, 
        \[ \slope(b_1 + b_2 + b_3,\eta) >0. \]

    As a first step, consider the slope of $b_1 + b_2$ along $\eta$. 
    Since $\rho(\sigma_1)$ and $\rho(\sigma_2)$ are opposite, the triangle inequality gives
    \begin{equation*}
        \pi = \angle(\rho(\sigma_1),\rho(\sigma_2)) \le \angle(\rho(\sigma_1),\eta) + \angle(\eta, \rho(\sigma_2))
    \end{equation*}
    which implies     
    \begin{equation*}
        \slope(b_1+b_2,\eta)  =  - \cos \angle(\rho(\sigma_1),\eta) - \cos \angle(\eta,\rho(\sigma_2)) \ge 0
    \end{equation*}
    with equality exactly when $\eta$ lies on a geodesic segment from $\rho(\sigma_1)$ to $\rho(\sigma_2)$ and hence lies in the boundary of the unique maximal flat $F(\sigma_1,\sigma_2)$ containing the chambers corresponding to $\sigma_1$ and $\sigma_2$.
    The same holds with $\{1,2\}$ replaced by $\{1,3\}$ or $\{2,3\}$.
        
    Then the sum
        \[ \slope(b_1+b_2+b_3,\eta) = \frac12 \slope\left((b_1+b_2)+(b_2+b_3)+(b_1+b_3),\eta\right) \ge 0. \]
    Suppose for the sake of contradiction that this sum is $0$. 
    The previous paragraph implies each of the terms is zero and $\eta$ lies in the boundary of all three flats $F(\sigma_1,\sigma_3),F(\sigma_2,\sigma_3),F(\sigma_1,\sigma_2)$.  

    Up to applying an element $g \in \PGL_d(\R)$, we may assume that $\sigma_1= \sigma_0$, $\sigma_2=n \sigma_0,$ and $\sigma_3=\sigma_\infty$ where $n$ is totally positive, and now $\eta$ lies in the boundary of the standard maximal flat $F_{std} = F(\sigma_0,\sigma_\infty)$. 
    Then $\eta$ spans a simplex $\tau$ in the boundary of $F_{std}$ and we may consider $v$ any vertex of $\tau$. 
    Since $\eta$ also belongs to the boundary of $F_{std}$, the same holds for $v$. 
    
    We claim that $v$ is opposite to a vertex of $\sigma_2$.
    To see this, observe that $v$ corresponds to a $(d-k)$-plane spanned by $(d-k)$ standard basis vectors $\{e_{i_1},\dots,e_{i_{d-k}}\}$. 
    The assertion that $v$ is opposite to a vertex of $\sigma_2=n \sigma_0$ amounts to saying that the minor consisting of the last $k$ columns of $n$ and the $k$ rows $\{1,\dots,d\} \setminus \{i_1,\dots,i_{d-k}\}$ is nonzero; but this is immediate from the total positivity of $n$.
    
    On the other hand, each vertex opposite to $\sigma_2$ in $F(\sigma_1,\sigma_2)$ is incident to $\sigma_1$. 
    Running the same argument with $\{1,3\}$ replaced by $\{2,3\}$ shows that $v$ is opposite to $\sigma_1$, yielding a contradiction.
\end{proof}

The second lemma is quite general.

\begin{lemma}\label{lem: convex proper functions}
    Let $\mathcal{S}$ be a set of convex $\R^{\ge 0}$-valued functions on a proper geodesic space $Y$, which is compact with respect to local uniform convergence. 
    Suppose that each member of $S$ is proper. 
    Then the total function on $\mathcal{S} \times Y$ is proper. 
\end{lemma}

\begin{proof}
    Fix $y \in Y$, and $C \geq \sup_{s \in S} s(y)$. Suppose there were a divergent sequence $\{s_n, y_n\}$ with $s_n(y_n) \leq C$. Let $s$ be a subsequential limit of $s_n$. We will contradict the properness of $s$.

    Fix any $r > 0$. For $n$ large enough that $d_Y(y,y_n) \geq r$, we can find a point $w_n$ on $\partial B_r(y)$ such that $s_n(w_n) \leq C$, namely a point on a geodesic segment from $y$ to $y_n$. Let $w_r$ be a subsequential limit of $w_n$. Since $s_n \to s$ locally uniformly, $s_n(w_n) \to s(w_r)$. Hence $s(w_r) \leq C$. Hence the sublevel set of $s$ meets $\partial B_r(y)$ for each $r$, so it is not compact in $Y$.
\end{proof}

\begin{proof}[Proof of Theorem \ref{thm: purqies with same boundary}]

    Let $x_0$ be a point in $\H^2$ at distance at most 2 from the three geodesics $[0,\infty]$, $[0,1]$, and $[1,\infty]$. Up to isometry, we may assume that $f_1(x_0) = o$ for a given point $o \in Y$, and that $d_Y(f_1(x_0),f_2(x_0))$ is within 1 of its maximum. Let $\mathcal{S}$ be the space of all possible triples of flags $(\phi(0), \phi(1), \phi(\infty))$, where $\phi$ ranges over the boundary maps of all $(L,M,\Pi)$-quasi-isometric embeddings taking $x_0$ to $o$. By Lemma \ref{lem: compact triples}, $\mathcal{S}$ is a compact subset of the space of all positive triples of flags. Using the Weyl vector $\rho$, we may identify $\mathcal{S}$ with the corresponding set of convex functions
    \[
    s(y) = b_{\rho(\phi(0))}(y) + b_{\rho(\phi(1))}(y) + b_{\rho(\phi(\infty))}(y).
    \]
    Up until now, we have not made explicit the choice of additive constant for the sum of Buseman functions. In order to apply Lemma \ref{lem: convex proper functions}, we now make a convenient choice. Specifically, for each pair $i,j \in \{0,1,\infty\}$, we define $b_{ij} = b_{\rho(\phi(i))} + b_{\rho(\phi(j))}$ normalized such that its minimum value is zero, and we define
    \[
    s = \frac{1}{2}(b_{01} + b_{1\infty} + b_{0\infty}).
    \]
    Each such function $s$ is proper on $Y$ by Lemma \ref{lem: positive triples are stable}. Hence by Lemma \ref{lem: convex proper functions}, for any $C \in \R$, the set of $\{s \in \mathcal{S}, w \in Y\}$ such that $s(w) \leq C$ is compact. Since $\mathcal{S}$ is compact as well, the set of $w \in Y$ such that there exists some $s \in \mathcal{S}$ with $s(w) \leq C$ is also compact in $Y$.
    
    Since Busemann functions are 1-Lipschitz, and $b_{\rho(\phi(0))}(y) + b_{\rho(\phi(1))}(y)$ vanishes for $y$ in the flat $F(\phi(0),\phi(1))$, we have
    \[
    b_{\rho(\phi(0))}(y) + b_{\rho(\phi(1))}(y) \leq 2 d_Y(y,F(\phi(0),\phi(1)).
    \]
    Summing over the three pairs of flags gives
    \[
    s(y) \leq d_Y(y,F(\phi(0),\phi(1)) + d_Y(y,F(\phi(1),\phi(\infty)) + d_Y(y,F(\phi(1),\phi(\infty)).
    \]
    But the higher rank Morse lemma bounds the distance of $f_2(x_0)$ to each of the three flats e.g.\ $F(\partial f_2(0), \partial f_2(1))$ by $D + L(2+1)$. Together with the previous inequality, this bounds $s(f_2(x_0))$, and hence $f_2(x_0)$ must lie in a fixed compact subset of $Y$.

\end{proof}

\subsubsection{Existence}

Let $\phi: \RP^1 \to \mathrm{Flag}(\R^d)$ be a positive map. 
The $\PGL_d(\R)$-equivariant projection $p_d$ (Definition \ref{def: pd}) induces a map 
\[
p_d \circ \phi^3: (\RP^1)^{(3)} \to \mathrm{Flag}(\R^d)^{(3)+} \to Y.
\]

On the other hand, $p_2: (\RP^1)^{(3)} \to \H^2$ naturally identifies $(\RP^1)^{(3)}$ with the unit tangent bundle of $\H^2$, and we metrize $(\RP^1)^{(3)}$ correspondingly. 

\begin{lemma}
    There is an $L(d,K)$ such that if $\phi$ is $K$-quasisymmetric, then $p_d \circ \phi^3$ is $L$-coarse Lipschitz.
\end{lemma}

\begin{proof}
    By the triangle inequality, a map $f: X \to Y$ between geodesic metric spaces is $L$-coarse Lipschitz as long as
    \begin{equation} \label{eqn: coarse Lipschitz 1}
    d_Y(f(x),f(z)) \leq L \textrm{ whenever } d_X(x,z) \leq 1.
    \end{equation}
    Let $x$ and $z$ be two points of $(\RP^1)^{(3)}$ at distance at most 1. Up to the action of $\mathrm{PGL}_2(\R) \times \mathrm{PGL}_d(\R)$, we may suppose that $x = (0,1,\infty)$ and $\phi$ is normalized. Since $K$-quasisymmetric maps are quasiperiodic by Proposition \ref{prop: compactness of qs maps}, Lemma \ref{lem: compact triples} gives that $\phi^3(z)$ lies in a given compact subspace of the space of flags, and since $p_d$ is continuous this establishes \eqref{eqn: coarse Lipschitz 1} for some $L$.
\end{proof}

Since the fibers of $p_2$ are bounded diameter, $p_2$ is a quasi-isometric embedding. Therefore, any section $s: \H^2 \to (\RP^1)^{(3)}$ gives a coarse inverse, and the composition $p_d \circ \phi^3 \circ s$ is still $L$-coarse Lipschitz for a (possibly new) $L$ not depending on $s$. Fix any $s$.

\begin{lemma}
    There exist $L=L(K,d)$ and $M=M(K,d)$ depending only on $K$ and $d$ such that the map $p_d \circ \phi^3 \circ s$ is an $(L,M,\Pi)$-quasi-isometric embedding.
\end{lemma}

\begin{proof}
    Let $x_1$ and $x_2$ be any two distinct points of $\H^2$. 
    Up to the action of $\PGL_2(\R) \times \PGL_d(\R)$, we may assume that $x_1 = p_2(0,1,\infty)$, $x_2 = p_2(0,t,\infty)$ for some $t$ with $-\infty < t < -1$, and $\phi$ is normalized.

    Then $p_d(\sigma_0,\phi(1),\sigma_\infty)$ and $p_d(\sigma_0,\phi(t),\sigma_\infty)$ lie in the standard flat $F$ of $Y$. 
    We introduce coordinates $F \cong \R^{d-1}$ by first identifying $F$ with $\mathfrak{a}$ by the exponential, and then applying the $d-1$ simple roots to $\mathfrak{a}$.

    Then by Proposition \ref{prop: cross ratios and projections}, in these coordinates,
    \[
        p_d(\sigma_0,\phi(t),\sigma_\infty) = [cr_i(\phi(t), \sigma_0,\phi(1),\sigma_\infty)]_{i=1,\ldots, d-1}
    \]
    By Proposition \ref{prop: cr is coarse linear}, each of these is coarsely linear of positive slope.
    By the previous lemma, any choice of section $s$ only disturbs this growth by a bounded amount.
\end{proof}

This completes the proof of existence.

\bibliographystyle{plain}
\bibliography{bibliography}

\end{document}